\documentclass[a4paper,12pt]{amsart}
\usepackage{graphicx} \usepackage{amsmath}
\usepackage{latexsym,lscape}
\usepackage{amsthm,float}
\usepackage{autograph,epic,latexsym,bezier,caption,amsbsy,psfrag,color,enumerate,amsfonts,amsmath,amscd,amssymb}
\usepackage{epsfig}
\usepackage{rotating}
\usepackage{graphics}
\usepackage[latin1]{inputenc}

\setloopdiam{17}\setprofcurve{10}
\newtheorem{defi}{Definition}[section]
\newtheorem{theorem}[defi]{Theorem}
\newtheorem{lemma}[defi]{Lemma}
\newtheorem{prop}[defi]{Proposition}
\newtheorem{corollary}[defi]{Corollary}

\newtheorem{example}[defi]{Example}

\newtheorem{remark}[defi]{Remark}

\begin{document}

\title{Shuffling matrices, Kronecker product and Discrete Fourier Transform}

\author{Daniele D'Angeli}
\address{Daniele D'Angeli, Institut f\"{u}r Diskrete Mathematik, Technische Universit\"{a}t Graz \ \ Steyrergasse 30, 8010 Graz, Austria}
\email{dangeli@math.tugraz.at}
\author{Alfredo Donno}
\address{Alfredo Donno, Universit\`{a} degli Studi Niccol\`{o} Cusano - Via Don Carlo Gnocchi, 3 00166 Roma, Italia \\
Tel.: +39 06 45678356, Fax: +39 06 45678379}
\email{alfredo.donno@unicusano.it\qquad \textrm{(Corresponding Author)}}

\maketitle

\begin{abstract}
We define and investigate a family of permutations matrices, called shuffling matrices, acting on a set of $N=n_1\cdots n_m$ elements, where $m\geq 2$ and $n_i\geq 2$ for any $i=1,\ldots, m$. These elements are identified with the vertices of the $m$-th level of a rooted tree with branch indices $(n_1,\ldots, n_m)$. Each of such matrices is induced by a permutation of $Sym(m)$ and it turns out that, in the case in which one considers the cyclic permutation $(1\ \ldots\ m)$, the corresponding permutation is the classical perfect shuffle. We give a combinatorial interpretation of these permutations in terms of lexicographic order of the vertices of the tree. This allows us to describe their fixed points. We show that our permutation matrices can be used to let the Kronecker product of matrices commute or, more generally, rearrange in an arbitrary order. Moreover, we show that the group generated by such permutations does depend only on the branch indices of the tree, but it is independent from their order. In the case in which such indices coincide, we prove that the corresponding group is a copy of $Sym(m)$ inside $Sym(n^m)$. Finally, we give an application of shuffling matrices in the context of the Discrete Fourier Transform.
\end{abstract}

\begin{center}
\end{center}

\section{Introduction}

A perfect shuffle is a very natural way of permuting $2n$ cards of a deck. The deck is divided in two parts and then the cards are reordered by interleaving the two decks in one of the two possible ways: one leaving the original top card on top (classical model), one leaving the original top card second from the top. This operation has a very easy mathematical description: one can number the cards from 0 to $2n-1$ and write the permutation obtained after performing the shuffle. Diaconis, Grahm and Kantor \cite{diaconis} were able to determine the structure of the group generated by the two permutations arising from the two perfect shuffles. What is interesting, it is the fact that Diaconis has been a magician and knew many card tricks based on perfect shuffles. If a talented magician is able to perform a perfect shuffle, he knows how the cards are distributed in the deck after the shuffle. The paper \cite{diaconis} contains a very nice section about the history of the perfect shuffle. We refer to that paper and references therein for a historical account on this interesting aspect. It is worth mentioning here the huge literature about shuffling cards and related probabilistic topics \cite{aldousdiaconis, bayerdiaconis, diaconis2}.

Rose \cite{rose} generalized this shuffle by interleaving in a suitable way two sets of size $r$ and $s$. The idea is very simple: consider $rs$ objects grouped in $r$ piles of size $s$ so that they are disposed in an array with $s$ rows and $r$ columns. Then rearrange them in $s$ piles of size $r$ just by switching the rows with the columns. He used the permutation matrices arising from the perfect shuffle of two sets in order to give a method to compute the Discrete Fourier Transform, with an interesting application to the algorithm called Fast Fourier Transform (FFT). Fast Fourier Transforms are widely used for many applications in engineering, science, and mathematics and were described by Gilbert Strang as ``the most important numerical algorithm of our lifetime" \cite{strang}.
Rose noticed that its shuffle was somehow related to the Kronecker (tensor) product of two matrices. This product is, in general, non-commutative, but the commutation can be achieved up to multiplying with opportune permutation matrices of the shuffle. This correspondence was further generalized and exploited by Davio \cite{davio}, who used simple properties of the algebra generated by the matrices arising from the shuffle in order to describe a wide class of switching circuits. Davio generalized the results of Rose by considering $m$ subsets instead of just two, and used the mixed radix formalism for applications in network design. In \cite{ronse}, Ronse generalized further this construction by introducing generalized shuffling permutations, which contain the cases already studied by Rose and Davio, as particular cases.

Our paper can be framed into the context proposed by Davio and Ronse. More precisely, we define a family of permutation matrices, called \textit{shuffling matrices}, associated with $m$ integers $n_1, \ldots, n_m$ and a permutation $\sigma\in Sym(m)$: these matrices will be denoted by $P^{\sigma}_{n_1,\ldots,n_m}$. We investigate the corresponding shuffling permutations on $N=n_1\cdots n_m$ elements, their action on rooted trees, and we study shuffling matrices in terms of their action on iterated Kronecker products of matrices, the conjugacy problem, and an explicit application to the Discrete Fourier Transform theory. More specifically, we consider $m$ sets $X_1,\ldots, X_m$ of size $n_1,\ldots, n_m$, respectively, and identify their cartesian product with the vertices of the $m$-th level of a rooted tree whose branch indices are exactly $n_1,\ldots, n_m$. The $N$ elements of the product can be identified with words of length $m$, whose $i$-th letter belongs to the alphabet $X_i$. The $N$ elements are naturally ordered lexicographically, so that the first letter (or coordinate) can be seen as the most important, and the $m$-th letter as the less important for the ordering. The motivation for this paper is the following observation: if we permute cyclically the $m$ coordinates, by meaning that we permute cyclically their order of importance and consider the lexicographic order given by this rearrangement, we get a permutation which coincides with the perfect shuffle studied by Davio. This reduces to the cases of Rose (for $m=2$) and of Diaconis, Grahm and Kantor (for $m=2$ and $n_2=2$). This suggests the idea that one can study the properties usually studied for the perfect shuffle matrices, when the full symmetric group $Sym(m)$ (and not only cyclic permutations) is considered. Such \lq\lq generalized shuffling permutations\rq\rq were studied by Ronse in \cite{ronse} by using the mixed radix formalism. In our paper, we adopt a matrix approach and we let these permutations act on rooted trees.  This analysis leads to a number of natural questions. We were able to rephrase the results given by Rose and Davio in this more general setting and discussed some algebraic properties of the group of permutations on $N$ objects generated by shuffling matrices. In the homogeneous case $n_1=\ldots =n_m=n$, this enables us to explicitly construct a subgroup of $Sym(n^m)$ isomorphic to $Sym(m)$.
We plan to study, in a future work, products of shuffling matrices and their action on trees or more general combinatorial structures (in the spirit of \cite{dado2, dado3}), and to investigate the cutoff phenomenon for the associated Markov chains \cite{diaconis3, dado, figa}.

The paper is structured as follows. In Section \ref{sectionpreliminaries}, we introduce some basic notation and definitions about mixed radix representation of integers, rooted trees and branch indices, Kronecker products, permutation matrices, and we recall some properties and interpretations of the classical perfect shuffle studied by Rose and Davio. In Section \ref{mainsection}, we introduce the central definition of the paper, represented by the notion of shuffling matrix, and we study some combinatorial and algebraic properties of such permutation matrices (Section \ref{sectionmaindefinition}); in Section \ref{subsectionconjugacy}, we focus our interest on the conjugacy property. In Section \ref{sectiongruppi}, the group generated by shuffling permutations is investigated; in particular, the Section \ref{more} is devoted to the study of the group consisting of perfect shuffle permutations. Finally, the Section \ref{sectionfourier} describes an explicit application of the shuffling matrices to the Fast Fourier Transform theory. The main results can be summarized as follows:\\
$-$ in Proposition \ref{propositionaction}, the image of any integer $x$ satisfying $0\leq x\leq N-1$ under the action of the permutation  associated with the shuffling matrix $P^{\sigma}_{n_1,\ldots,n_m}$ is explicitly described;\\
$-$ in Theorem \ref{theoremconiugio}, we prove that shuffling matrices are able to rearrange the factors of an iterated Kronecker product of $m$ matrices, according with any new order induced by a permutation $\sigma\in Sym(m)$; in the square case, this action is achieved by conjugation, as shown in Theorem \ref{coro1coniugio};\\
$-$ in Theorem \ref{teo:gruppo non omogeneo}, we prove that the group generated by shuffling permutations does not depend on the order of the branch indices, and in Corollary \ref{corollario:gruppo omogeneo} we deduce that, in the homogeneous case, such group is isomorphic to the symmetric group $Sym(m)$;\\
$-$ in Proposition \ref{gruppodeglishuffling}, we show that the group of perfect shufflings on $N$ elements is isomorphic to the multiplicative subgroup of invertible elements of $\mathbb{Z}_{N-1}$;\\
$-$ in Theorem \ref{generaltheoremfourier}, we explicitly describe the block decomposition of the Discrete Fourier Matrix under the action by multiplication of a shuffling permutation matrix: such computation can be considered as the basis of an extended Fast Fourier Transform algorithm.

\section{Preliminaries}\label{sectionpreliminaries}

Rose studied in \cite{rose} the perfect shuffle on a set of $n=rs$ elements. He proposed to represent the elements $x_0,x_1,\ldots, x_{sr-1}$ of such set in an $s\times r$ array
$$
X = (X_{ij}) = \left(
      \begin{array}{cccc}
        x_0 & x_1 & \cdots & x_{r-1} \\
        x_r & x_{r+1} & \cdots & x_{2r-1} \\
        \vdots &  &  & \vdots \\
        x_{(s-1)r} & \cdots & \cdots & x_{sr-1} \\
      \end{array}
    \right)
$$
according with the (unique) representation of each integer in $\{0,1,\ldots, rs-1\}$ as
$$
ir+j, \qquad 0\leq i \leq s-1, \qquad 0\leq j\leq r-1,
$$
so that $X_{ij}=x_{ir+j}$. The perfect shuffle consists in switching the role of $r$ and $s$, by passing to the representation
$$
i's+j', \qquad 0\leq i' \leq r-1, \qquad 0\leq j'\leq s-1,
$$
corresponding to the array
$$
Y = (Y_{i'j'}) =  \left(
      \begin{array}{cccc}
        x_0 & x_r & \cdots & x_{(s-1)r} \\
        x_1 & x_{r+1} & \cdots & x_{(s-1)r+1} \\
        \vdots &  &  & \vdots \\
        x_{r-1} & x_{2r-1} & \cdots & x_{sr-1} \\
      \end{array}
    \right)=X^T,
$$
with $Y_{i'j'} = y_{i's+j'}$, where $y$ is the vector of length $n$ obtained from the vector $x=(x_0,x_1,\ldots, x_{r-1}, x_r,\ldots, x_{sr-1})$ by taking the entries of $X$ column by column. The transformation from the vector $x$ to the vector $y$ can be performed by an $n\times n$ permutation matrix denoted by $P_r^s$, that is, $y=P^s_rx$. The other way around is clearly obtained by the matrix $P^r_s=(P^s_r)^{-1}$.
Rose gives an explicit description of such matrices and shows that they naturally appear in changing the order of the factors in the Kronecker product of two matrices (see \cite{rose}, Proposition 1): given an $r\times r$ matrix $R$
and an $s\times s$ matrix $S$, it holds:
$$
R\otimes S= P^s_r (S\otimes R) P^r_s.
$$
Davio generalized the construction of Rose to the case in which $n$ is the product of more than two factors, that is, $n = b_{m-1}b_{m-2}\cdots b_0$, by using the mixed radix representation of integers in $\{0,1,\ldots, n-1\}$ with respect to the basis vector $[b_{m-1}, b_{m-2}, \ldots, b_0]$. Keeping the analogy with the Rose strategy, we can think that he disposes the $n$ elements in a $b_{m-1}\times (b_{m-2}\cdots b_0)$ array. He describes how the perfect shuffle is perturbed when one performs a cyclic shift of the factors and he is able to give a matrix representation of this phenomenon (see Theorems 4 and 5 in \cite{davio} and Section \ref{mainsection} below). The paper of Davio was very influential in literature because of its applications: he used his results in order to give a simple description of many switching circuits. In fact, the shuffle can be interpreted as a rearrangement of electrical connections and he gave an interpretation of the formulae arising from the algebraic structure behind the shuffling in terms of compositions of circuits (see Figure \ref{davio_circuit}). The paper \cite{ronse} generalized further his construction, in connection with applications to switching devices.
\begin{figure}[h]\begin{center}
\includegraphics[width=0.3\textwidth]{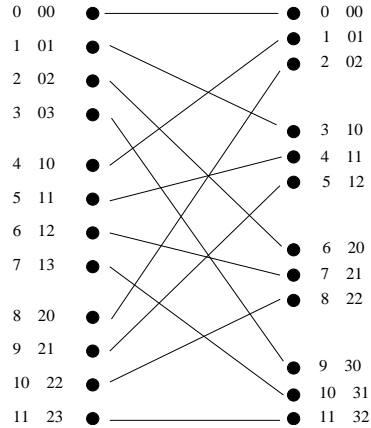}
\caption{The circuit representation of the perfect shuffle $3\times 4$.} \label{davio_circuit}
  \end{center}
\end{figure}

Notice that the $3\times 4$ elements of the set $\{0,1,\ldots, 11\}$ can be identified with the vertices of the second level of the rooted tree $T_{3,4}$ of branch indices $(3,4)$ represented in Figure \ref{figuraaggiuntina}. In the first row, we can read the representation of the integers with respect to the basis vector $[3,4]$; in the second row, they are listed in increasing order from $0$ to $11$; in the third row, we read the action of the shuffle permutation.

\begin{figure}[h]\begin{center}
\includegraphics[width=0.8\textwidth]{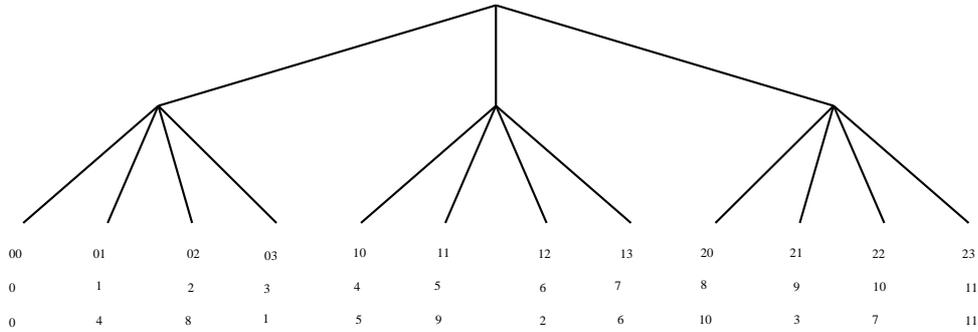}
\caption{The tree representation of the perfect shuffle $3\times 4$.} \label{figuraaggiuntina}
  \end{center}
\end{figure}

In this tree interpretation of the perfect shuffle, we can think that, from the subtree of depth $1$ rooted at each vertex of the first level of $T_{3,4}$, we take the first element (that is, $0,4,8$), then the second element (that is, $1,5,9$), then the third element (that is, $2,6,10$) and finally the fourth element (that is, $3,7,11$), obtaining in this way the final configuration in the third row of Figure \ref{figuraaggiuntina}.

In order to generalize the constructions by Rose and Davio, we start by introducing some notation. Let $m\geq 2$ be a natural number and, for each $i =1,\ldots, m$, let $X_i = \{0,1,\ldots, n_i-1\}$ be a finite alphabet of $n_i$ letters, with $n_i \geq 2$. We put $X = X_1 \times X_2 \times \cdots \times X_m$, so that an element $x=x_1x_2\ldots x_m$ of $X$ is a word of $m$ letters, whose $i$-th letter $x_i$ belongs to the alphabet $X_i$. Put $N = |X| = \prod_{i=1}^mn_i$.

Notice that, for every $j=1,\ldots, m$, the elements of the set $X_1\times X_2\times \cdots \times X_j$ can be naturally identified with the vertices of the $j$-th level of the finite rooted tree $T_{n_1,\ldots,n_m}$ of depth $m$, with branch indices $(n_1,\ldots, n_m)$. In particular, the elements of the set $X$ are identified with the $N$ vertices of the $m$-th level of $T_{n_1,\ldots,n_m}$. Moreover, they can be naturally listed in the lexicographic order as shown in the example of Figure \ref{exampletree}, where we have $m=3$ and $n_1=4$, $n_2=3$, $n_3=2$.

\begin{figure}[h] \begin{center}
\includegraphics[width=1\textwidth]{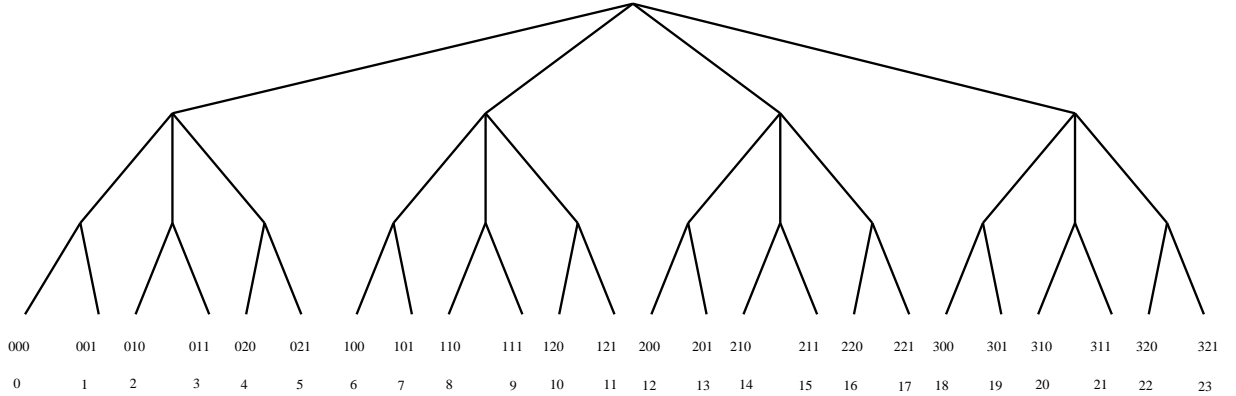} \caption{The rooted tree $T_{4,3,2}$.}  \label{exampletree}
 \end{center}
\end{figure}
It is worth mentioning that this identification and the associated lexicographic order agree with the classical notion of mixed radix representation of a nonnegative integer, which is recalled below.\\
Let $n_1, n_2, \ldots, n_m$, with $n_i \geq 2$ for every $i=1,\ldots, m$, and let ${\bf n} = [n_1, n_2, \ldots, n_m]$ be the associated \textit{basis vector}. Now put
$$
u_0 = 1;\qquad u_i = \prod_{j=0}^{i-1}n_{m-j}, \ \forall i=1,\ldots,m-1.
$$
It is known (see, for instance, \cite{knuth}) that, given any integer $x$ such that $0\leq x \leq N-1$, there exists a unique representation of $x$ with respect to the \textit{basis vector} ${\bf n} = [n_1, n_2, \ldots, n_m]$, or with respect to the \textit{weight vector} ${\bf u} = [u_{m-1}, u_{m-2}, \ldots, u_0]$, called \textit{mixed radix representation} of $x$ with respect to ${\bf n}$, or with respect to ${\bf u}$, respectively, given by
\begin{eqnarray*}
x &=& \sum_{j=1}^m x_j u_{m-j}\\
&=& x_1 n_2\cdots n_m + x_2 n_3 \cdots n_m + \cdots + x_{m-1}n_m + x_m,
\end{eqnarray*}
with $0\leq x_j\leq n_j-1$, for each $j=1,\ldots, m$.

According with the mixed radix representation, the vertex of the $m$-th level of $T_{n_1,\ldots,n_m}$ which is labelled by $x=x_1\ldots x_m$, with $x_i\in X_i$, is identified with the integer $\sum_{j=1}^m x_j u_{m-j}$ belonging to the set $\{0,1,\ldots, N-1\}$.

\begin{example}\rm
The vertices of the third level of the rooted tree $T_{4,3,2}$ depicted in Figure \ref{exampletree} can be identified, from the left to the right, with the integers from $0$ to $23$, listed in increasing order. For instance, the vertex $211$ corresponds to the integer $2\cdot 6 + 1\cdot 2 +1\cdot 1= 15$, since we have $u_2=3\cdot 2 =6$, $u_1=2$, $u_0=1$.
\end{example}

Let us denote by $\mathcal{M}_{m\times n}(\mathbb{R})$ the set of matrices with $m$ rows and $n$ columns over the real field. For every $i=1,\ldots, m$ and $j=1,\ldots, n$, let $E^{i,j}_{m\times n}$ denote the $m\times n$ elementary matrix, with $1$ at the entry $(i,j)$ at row $i$ and column $j$, and $0$ elsewhere. Moreover, let us denote by $O_{m\times n}$ the zero matrix with $m$ rows and $n$ columns.

We recall that the \textit{Kronecker product} of two matrices $A=(a_{ij})_{i=1,\ldots, m; j=1,\ldots, n}\in \mathcal{M}_{m\times n}(\mathbb{R})$ and $B=(b_{hk})_{h=1,\ldots, p; k=1,\ldots, q}\in \mathcal{M}_{p\times q}(\mathbb{R})$ is defined to be the $mp\times nq$ matrix
$$
A\otimes B = \left(
                \begin{array}{ccc}
                  a_{11}B & \cdots & a_{1n}B \\
                  \vdots & \ddots & \vdots \\
                  a_{m1}B & \cdots & a_{mn}B \\
                \end{array}
              \right).
$$
Finally, recall that the Kronecker product of two matrices is, in general, not commutative, that is, $A\otimes B \neq B\otimes A$. On the other hand, the Kronecker product satisfies the following properties:
\begin{enumerate}
\item $A\otimes (B\otimes C) = (A\otimes B)\otimes C$ (associativity);
\item $A\otimes (B+C) = A\otimes B + A\otimes C$; \quad $(A+B)\otimes C = A\otimes C + B\otimes C$ (distributivity);
\item $(A\otimes B)^T = A^T \otimes B^T$ (transposition);
\item $(A\otimes B)^{-1} = A^{-1} \otimes B^{-1}$ (inversion);
\item $(A\otimes B) (C\otimes D) = (AC)\otimes (BD)$ (standard matrix multiplication).
\end{enumerate}

\begin{lemma}\label{lemmaciccino}
Let $E^{i,j}_{m\times n}\in \mathcal{M}_{m\times n}(\mathbb{R})$ and $E^{h,k}_{n\times q}\in \mathcal{M}_{n\times q}(\mathbb{R})$ be two elementary matrices. Then
\begin{eqnarray*}
 E^{i,j}_{m\times n} \cdot E^{h,k}_{n\times q} = \left\{
                                                  \begin{array}{ll}
                                                    O_{m\times q} & \hbox{if }j\neq h \\
                                                    E^{i,k}_{m\times q}, & \hbox{if } j=h.
                                                  \end{array}
                                                \right.
\end{eqnarray*}
\end{lemma}
\begin{proof}
The proof is straightforward, and it uses the definition of elementary matrix and of standard matrix multiplication.
\end{proof}

\begin{lemma}\label{secondlemma}
Let $E^{i,j}_{m\times n}\in \mathcal{M}_{m\times n}(\mathbb{R})$ and $E^{h,k}_{p\times q}\in \mathcal{M}_{p\times q}(\mathbb{R})$ be two elementary matrices. Then
$$
E^{i,j}_{m\times n} \otimes E^{h,k}_{p\times q} = E^{(i-1)p+h,(j-1)q+k}_{mp\times nq}.
$$
More generally, we have
$$
E^{i_1,j_1}_{m_1\times n_1}\otimes\cdots \otimes E^{i_t,j_t}_{m_t\times n_t} = E^{\sum_{h=1}^{t-1}\left((i_h-1)\prod_{k=h+1}^tm_{k}\right)+i_t,\sum_{h=1}^{t-1}\left((j_h-1)\prod_{k=h+1}^tn_{k}\right)+j_t}_{\prod_{h=1}^tm_h \times \prod_{h=1}^tn_h}.
$$
\end{lemma}
\begin{proof}
It is an easy induction following from the definition of Kronecker product and from the associativity property.
\end{proof}
We will denote by $I_n$ the identity matrix of size $n$. Moreover, let ${\bf e}_j^n$ denote the column vector $(0,\ldots, 0,\underbrace{1}_{j-th\ place},0,\ldots,0)^T$ of length $n$, for every $j=1,\ldots, n$.

\begin{lemma}\label{lemmavettorini}
The following identities hold:
\begin{eqnarray*}
\sum_{i = 1,\ldots, n} ({\bf e}^{n}_{i})^T \otimes {\bf e}^{n}_{i} = I_{n};
\end{eqnarray*}
\begin{eqnarray*}
{\bf e}^{n}_{i} \otimes  ({\bf e}^{k}_{j})^T = E^{i,j}_{n\times k},
\end{eqnarray*}
for each $i=1,\ldots, n$ and $j=1,\ldots, k$.
\end{lemma}
\begin{proof}
It is an easy computation.
\end{proof}

In what follows, we will denote by $Sym(m)$ the symmetric group on $m$ elements. We recall that a \textit{permutation matrix} of size $m$ is an $m\times m$ matrix that has exactly one entry equal to $1$ in each row and each column and all $0$'s elsewhere. Note that each such matrix represents a permutation $\sigma$ of $m$ elements and this correspondence is clearly bijective. More precisely, the permutation $\sigma\in Sym(m)$ corresponds to the permutation matrix $P_{\sigma}=(p_{ij})_{i,j=1,\ldots, m}$ such that $p_{ij}=1$ if and only if $\sigma(j) = i$.\\ We will regard a permutation matrix as acting on column vectors of length $m$ whose entries are all $0$'s except that in one position, where the entry is $1$. More precisely, one has: $P_\sigma {\bf e}_j^m = {\bf e}_i^m$ if and only if $\sigma(j) = i$.

\section{Shuffling matrices}\label{mainsection}

In this section, we give the main definition of the paper, which extends the concept of perfect shuffle matrix studied by Rose and Davio. More precisely, we are going to define a family of permutation matrices, called shuffling matrices, obtained as linear combinations of iterated Kronecker products of elementary matrices. Such matrices are indexed by an $m$-tuple $n_1,\ldots, n_m$, with $n_i\geq 2$ for each $i$, and by a permutation $\sigma\in Sym(m)$; moreover, they can be regarded as acting on the vertices of the $m$-th level of a rooted tree $T_{n_1,\ldots, n_m}$ of branch indices $(n_1,\ldots, n_m)$ (see Section \ref{sectionmaindefinition}) or, equivalently, on vectors of length $N=n_1\cdots n_m$. The permutations associated with these matrices will be called shuffling permutations, and they were studied by Ronse in \cite{ronse} by merely using a mixed radix representation approach, without introducing permutation matrices. It turns that the choice $\sigma= (1\ 2\ \ldots\ m)$ of $Sym(m)$ induces the perfect shuffle on $N$ elements. Our set of permutation matrices has very nice properties that we discuss in the next sections. In particular, we highlight an interesting connection with the algebra induced by the Kronecker product of matrices (Section \ref{subsectionconjugacy}). We will examine the groups generated by shuffling permutations in Section \ref{sectiongruppi} and finally, in Section \ref{sectionfourier}, we will consider an interesting application to the Fast Fourier Transform.

\subsection{Definitions and basic properties}\label{sectionmaindefinition}

Our aim is to associate, with every permutation $\sigma \in Sym(m)$, a permutation matrix of size $N$. We will denote by $\widetilde{\sigma}$ the corresponding permutation of $Sym(N)$. Recall that the elements $\{0,1,2,\ldots, N-1\}$ can be identified with the vertices of the $m$-th level of the rooted tree $T_{n_1, n_2, \ldots, n_m}$ with branch indices $(n_1, n_2, \ldots, n_m)$, as described in Section \ref{sectionpreliminaries}.

We move from the following simple remark. The perfect shuffle permutation matrix $P_r^s$, introduced by Rose in order to describe the perfect shuffle of a set of $rs$ elements, can be rewritten as
$$
P_r^s = \sum_{\substack{i = 1,\ldots,r \\ j=1,\ldots, s}} E^{(i-1)s+j, (j-1)r+i}_{rs \times rs} =\sum_{\substack{i = 1,\ldots,r \\ j=1,\ldots, s}} E^{i,j}_{r\times s} \otimes  E^{j,i}_{s\times r} =  \sum_{i=1,\ldots, r}{\bf e}^r_i \otimes I_s \otimes ({\bf e}^r_i)^T,
$$
where the second equality follows from Lemma \ref{secondlemma} and the third equality from Lemma \ref{lemmavettorini}.

More generally, the following proposition describes the perfect shuffle on $N=n_1\cdots n_m$ elements (regarded as disposed in an $(n_1\cdots n_{m-1})\times n_m$ array).
\begin{prop}\label{proposizioneemme}
Let $T_{n_1,\ldots, n_m}$ be the rooted tree with branch indices $(n_1,\ldots, n_m)$. Then the permutation matrix
\begin{eqnarray}\label{osvaldo}
P=\sum_{i_m=1,\ldots, n_m}  {\bf e}^{n_m}_{i_m} \otimes I_{n_1\cdots n_{m-1}}\otimes ({\bf e}^{n_m}_{i_m})^T
\end{eqnarray}
induces the perfect shuffle on $N$ elements.
\end{prop}
\begin{proof}
First of all, notice that, by Lemma \ref{lemmavettorini} and Lemma \ref{secondlemma}, and by using the properties of the Kronecker product, the matrix \eqref{osvaldo} can be rewritten as
\begin{eqnarray}\label{seratasfumata}
P &=& \sum_{\substack{i_j = 1,\ldots,n_j \\ j=1,\ldots, m}} {\bf e}^{n_m}_{i_m} \otimes \left(({\bf e}^{n_1}_{i_1})^T\otimes {\bf e}^{n_1}_{i_1}\right) \otimes \cdots \otimes \left(({\bf e}^{n_{m-1}}_{i_{m-1}})^T\otimes {\bf e}^{n_{m-1}}_{i_{m-1}}\right) \otimes ({\bf e}^{n_m}_{i_m})^T\nonumber\\
&=&\sum_{\substack{i_j = 1,\ldots,n_j \\ j=1,\ldots, m}} E^{i_{m},i_1}_{n_{m}\times n_1} \otimes  E^{i_{1},i_2}_{n_{1}\times n_2} \otimes
\cdots \otimes E^{i_{m-1},i_m}_{n_{m-1}\times n_m}\nonumber\\
&=& E^{(i_m-1)n_1\cdots n_{m-1}+(i_1-1)n_2\cdots n_{m-1} +\cdots +(i_{m-2}-1)n_{m-1}+i_{m-1}, (i_1-1)n_2\cdots n_{m} +\cdots +(i_{m-1}-1)n_{m}+i_{m}}_{N\times N}.
\end{eqnarray}
Now let $({\bf e}^N_x)^T$ be the column vector with all $0$'s, except for an entry equal to $1$ at position
\begin{eqnarray*}
x &=& x_1n_2\cdots n_m + x_2n_3\cdots n_m + \cdots + x_{m-1}n_m+x_m\\
 &=& \underbrace{(x_1n_2\cdots n_{m-1}+\cdots +x_{m-2}n_{m-1}+x_{m-1})}_{i}n_m+\underbrace{x_m}_{j}.
\end{eqnarray*}
Then by using \eqref{seratasfumata} we get $P{\bf e}^N_x = {\bf e}^N_y$, with
$$
y = \underbrace{x_m}_{i'}n_1\cdots n_{m-1}+ \underbrace{x_1 n_2\cdots n_{m-1} + \cdots +x_{m-2}n_{m-1} + x_{m-1}}_{j'}.
$$
This shows that the matrix $P$ realizes the perfect shuffle of $N$ elements, regarded as disposed in an $(n_1\cdots n_{m-1})\times n_m$ array. In other words, we are passing from the configuration of the $N$ elements $\{a_i, i=0,\ldots, N-1\}$ given by the $(n_1\cdots n_{m-1})\times n_m$ array
$$
A = (A_{ij}) = \left(
      \begin{array}{cccc}
        a_0 & a_1 & \cdots & a_{n_m-1} \\
        a_{n_m} & a_{n_m+1} & \cdots & a_{2n_m-1} \\
        \vdots &  &  & \vdots \\
        a_{(n_1\cdots n_{m-1}-1)n_m} & \cdots & \cdots & a_{n_1\cdots n_m-1} \\
      \end{array}
    \right)
$$
with $A_{ij} = a_{in_m+j}$, for $0\leq i\leq n_1\cdots n_{m-1}-1$ and $0\leq j\leq n_m-1$, to the configuration $B=(B_{i'j'}) = A^T$, with $B_{i'j'} = b_{i'n_1\cdots n_{m-1}+j'}$, for $0\leq i' \leq n_m-1$ and $0\leq j'\leq n_1\cdots n_{m-1}-1$. This proves the assertion.
\end{proof}

\begin{example}\label{esempioemme}\rm
Let $m=3$; $n_1=n_2=2$ and $n_3=3$, so that we have $N=12$. The matrix
\begin{eqnarray*}
P &=& \sum_{\substack{i_j = 1,\ldots,n_j \\ j=1,2,3}} E^{i_3,i_1}_{3\times 2} \otimes E^{i_1,i_2}_{2\times 2} \otimes E^{i_2,i_3}_{2\times 3} = \sum_{i_3=1,2,3}{\bf e}^3_{i_3} \otimes I_4 \otimes ({\bf e}^3_{i_3})^T
\end{eqnarray*}
performs the perfect shuffle on $12$ elements, identified with the vertices of the third level of the tree $T_{2,2,3}$ as in Figure \ref{FigureExample}.
\begin{figure}[h] \begin{center}  \psfrag{000}{$000$}\psfrag{001}{$001$} \psfrag{002}{$002$} \psfrag{010}{$010$} \psfrag{011}{$011$} \psfrag{012}{$012$} \psfrag{100}{$100$} \psfrag{101}{$101$} \psfrag{102}{$102$} \psfrag{110}{$110$} \psfrag{111}{$111$} \psfrag{112}{$112$}

\psfrag{0}{$0$} \psfrag{1}{$1$} \psfrag{2}{$2$} \psfrag{3}{$3$} \psfrag{4}{$4$} \psfrag{5}{$5$} \psfrag{6}{$6$} \psfrag{7}{$7$} \psfrag{8}{$8$} \psfrag{9}{$9$} \psfrag{10}{$10$}   \psfrag{11111}{$11$} \psfrag{00}{$0$} \psfrag{11}{$3$} \psfrag{22}{$6$} \psfrag{33}{$9$} \psfrag{44}{$1$} \psfrag{55}{$4$} \psfrag{66}{$7$} \psfrag{77}{$10$} \psfrag{88}{$2$} \psfrag{99}{$5$} \psfrag{1010}{$8$}   \psfrag{1111}{$11$}

\psfrag{a0}{$000$} \psfrag{a1}{$010$} \psfrag{a2}{$100$} \psfrag{a3}{$110$} \psfrag{a4}{$001$} \psfrag{a5}{$011$} \psfrag{a6}{$101$} \psfrag{a7}{$111$} \psfrag{a8}{$002$} \psfrag{a9}{$012$} \psfrag{a10}{$102$}   \psfrag{a11}{$112$}
\includegraphics[width=0.7\textwidth]{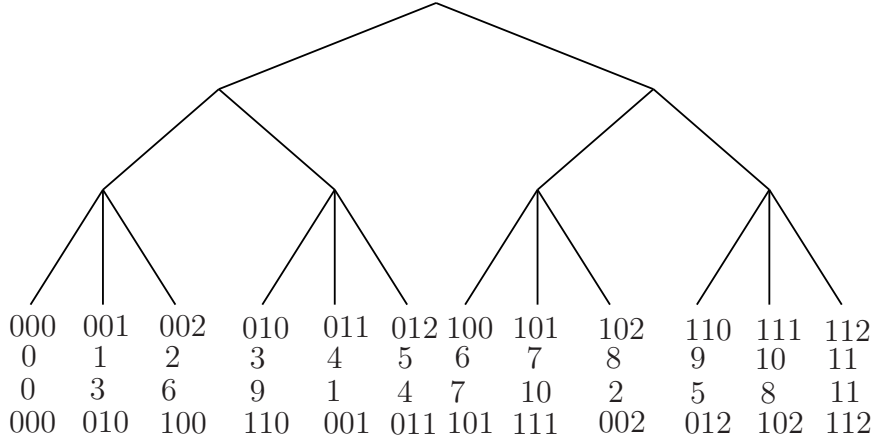} \caption{Example of perfect shuffle for $m=3$; $n_1=n_2=2, n_3=3$.}  \label{FigureExample}
\end{center}
\end{figure}
We say that the \textit{horizontal representation} of the permutation associated with $P$ is
$$
0\ 3\ 6\ 9\ 1\ 4\ 7\ 10\ 2\ 5\ 8\ 11;
$$
this corresponds to the permutation mapping $0$ to $0$, $1$ to $4$, $2$ to $8$, and so on. Moreover, its cyclic decomposition is $(0)(1\ 4\ 5\ 9\ 3)(2\ 8\ 10\ 7\ 6)(11)$, as it can be easily verified.
In the tree interpretation, we can think that, from the subtree of depth $1$ rooted at each vertex of the second level of $T_{2,2,3}$, we take the first element (that is, $0,3,6,9$), then the second element (that is, $1,4,7,10$), finally the third element (that is, $2,5,8,11$).
\end{example}

This analysis motivates to the following definition of shuffling matrix.

\begin{defi}\label{definitionmaindefinition}
The shuffling matrix $P^{\sigma}_{n_1, n_2, \ldots, n_m}$ associated with the permutation $\sigma\in Sym(m)$ and the branch indices $(n_1,\ldots, n_m)$ is the $N\times N$ permutation matrix
\begin{eqnarray}\label{maindefinition}
P^{\sigma}_{n_1, n_2, \ldots, n_m} = \sum_{\substack{i_j = 1,\ldots,n_j \\ j=1,\ldots, m}} E^{i_{\sigma^{-1}(1)},i_1}_{n_{\sigma^{-1}(1)}\times n_1} \otimes \cdots \otimes E^{i_{\sigma^{-1}(m)},i_m}_{n_{\sigma^{-1}(m)}\times n_m}.
\end{eqnarray}
\end{defi}

\begin{example}\label{esempionumerato}\rm
Let $m=3$; $n_1=n_2=2$ and $n_3=3$, so that we have $N=12$. Take the permutation $\sigma=(13)\in Sym(3)$. Then:
\begin{eqnarray*}
P^{(13)}_{2,2,3} &=& \sum_{\substack{i_j = 1,\ldots,n_j \\ j=1,2,3}} E^{i_3,i_1}_{3\times 2} \otimes E^{i_2,i_2}_{2\times 2} \otimes E^{i_1,i_3}_{2\times 3} = \sum_{\substack{i_1 = 1,2 \\ i_3=1,2,3}} E^{i_3,i_1}_{3\times 2} \otimes I_2 \otimes E^{i_1,i_3}_{2\times 3}.
\end{eqnarray*}
We have represented the action of the permutation $\widetilde{\sigma}$ induced by $P^{(13)}_{2,2,3}$ on the third level of $T_{2,2,3}$ in Figure \ref{FigureExamplepippino}.

\begin{figure}[h] \begin{center}  \psfrag{000}{$000$}\psfrag{001}{$001$} \psfrag{002}{$002$} \psfrag{010}{$010$} \psfrag{011}{$011$} \psfrag{012}{$012$} \psfrag{100}{$100$} \psfrag{101}{$101$} \psfrag{102}{$102$} \psfrag{110}{$110$} \psfrag{111}{$111$} \psfrag{112}{$112$}

\psfrag{0}{$0$} \psfrag{1}{$1$} \psfrag{2}{$2$} \psfrag{3}{$3$} \psfrag{4}{$4$} \psfrag{5}{$5$} \psfrag{6}{$6$} \psfrag{7}{$7$} \psfrag{8}{$8$} \psfrag{9}{$9$} \psfrag{10}{$10$}   \psfrag{11111}{$11$}
\psfrag{sigma}{$\widetilde{\sigma}$}
\psfrag{00}{$0$} \psfrag{11}{$6$} \psfrag{22}{$3$} \psfrag{33}{$9$} \psfrag{44}{$1$} \psfrag{55}{$7$} \psfrag{66}{$4$} \psfrag{77}{$10$} \psfrag{88}{$2$} \psfrag{99}{$8$} \psfrag{1010}{$5$}   \psfrag{1111}{$11$}

\psfrag{a0}{$000$} \psfrag{a1}{$100$} \psfrag{a2}{$010$} \psfrag{a3}{$110$} \psfrag{a4}{$001$} \psfrag{a5}{$101$} \psfrag{a6}{$011$} \psfrag{a7}{$111$} \psfrag{a8}{$002$} \psfrag{a9}{$102$} \psfrag{a10}{$012$}   \psfrag{a11}{$112$}
\includegraphics[width=0.7\textwidth]{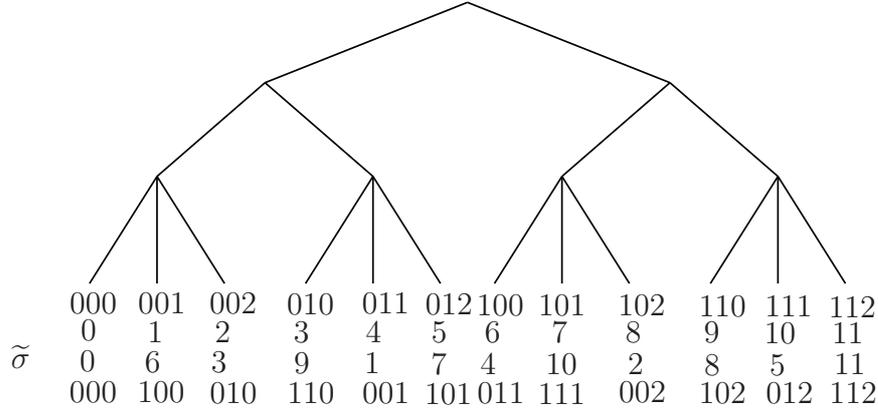} \caption{Shuffling permutation for $m=3$; $n_1=n_2=2, n_3=3$; $\sigma =(1\ 3)$.}  \label{FigureExamplepippino}
 \end{center}
\end{figure}
The horizontal representation of $\widetilde{\sigma}$ is $0\ 6\ 3\ 9\ 1\ 7\ 4\ 10\ 2\ 8\ 5\ 11$; its cyclic decomposition is $\widetilde{\sigma} = (0)(1\ 4\ 6)(2\ 8\ 9\ 3)(5\ 10\ 7)(11)$.
\end{example}

\begin{remark}\label{remarknuovissimo}\rm
Notice that, if we let the same permutation $\sigma=(1\ 3)\in Sym(3)$ act on the rooted tree $T_{2,3,2}$, so that $n_1=n_3=2$ and $n_2=3$ in this case, then we get the matrix permutation $P^{(13)}_{2,3,2} = \sum_{i_1,i_3=1,2} E^{i_3,i_1}_{2\times 2} \otimes I_3 \otimes E^{i_1,i_3}_{2\times 2}$, inducing the permutation $\widetilde{\sigma}\in Sym(12)$ whose cyclic decomposition is $\widetilde{\sigma} = (0)(1\ 6)(2)(3\ 8)(4)(5\ 10)(7)(9)(11)$, showing that the permutation of $N$ elements induced by $\sigma\in Sym(m)$ does depend on the order of the branch indices $(n_1, \ldots, n_m)$.
\end{remark}

It can be seen that, in the example in Figure \ref{FigureExample}, the elements of the third level of the tree $T_{2,2,3}$ are reordered by $\widetilde{\sigma}$ according with the following interpretation: in the starting configuration, the lexicographic order is indexed by the branch indices $(n_1,n_2,n_3)$, so that when we read the elements from the left to the right we have that the first coordinate which changes is the third one, then we have the second one, finally the first one: in fact, the first coordinate is equal to $0$ for the first half of the elements. Alternatively, we can think that, in the mixed radix representation of the integer $x = x_1x_2x_3\in \{0,\ldots,11\}$, the \lq\lq most important\rq\rq coordinate is $x_1$, then we have $x_2$ and finally $x_3$, as the coordinate $x_3$ is the first to be changed.\\ In the new configuration of the twelve elements under the action of $\widetilde{\sigma}$, with $\sigma=(1\ 2\ 3)$, we have, on the other hand, that the integers from $0$ to $11$ are rearranged according with a lexicographic order indexed by the branch indices $(n_3,n_1,n_2)$, so that now the third coordinate is the most important, then we have the first coordinate and finally the second coordinate becomes the less important.\\ A similar remark can be done in the case of Figure \ref{FigureExamplepippino}, where the most important coordinate after the action is the third coordinate, then we have the second coordinate, finally the first coordinate.

These are not special cases, as the analogous property holds in the general case. More specifically, the following proposition holds.

\begin{prop}\label{propositionaction}
Let $T_{n_1,\ldots, n_m}$ be the rooted tree with branch indices $(n_1,\ldots, n_m)$, and let $\sigma\in Sym(m)$. Let, as usual, $N = \prod_{i=1}^mn_i$. Order the integers from $0$ to $N-1$ in the lexicographic order induced by the mixed radix representation with respect to the basis vector $[n_1, n_2, \ldots, n_m]$, so that the most important coordinate of the element $x=x_1x_2\ldots x_m$ is $x_1$, and the less important coordinate of $x$ is $x_m$.
Then the image of the vertex of the $m$-th level of $T_{n_1,\ldots, n_m}$, which is identified with the integer $x = \sum_{j=1}^m x_j u_{m-j}$, under the action of $\widetilde{\sigma}$ is
\begin{eqnarray*}
\widetilde{\sigma}(x) &=& \sum_{j=1}^{m}x_{\sigma^{-1}(j)} \sigma^{-1}(u_{m-j}),
\end{eqnarray*}
with $\sigma^{-1}(u_i) = \prod_{j=0}^{i-1}n_{\sigma^{-1}(m-j)}$ for every $i=1,\ldots, m-1$ and $\sigma^{-1}(u_0) = u_0 =1$. More explicitly:
\begin{eqnarray*}
\widetilde{\sigma}(x) &=& x_{\sigma^{-1}(1)} n_{\sigma^{-1}(2)} \cdots n_{\sigma^{-1}(m)} + x_{\sigma^{-1}(2)} n_{\sigma^{-1}(3)} \cdots n_{\sigma^{-1}(m)} +\cdots\\
&+& x_{\sigma^{-1}(m-1)}n_{\sigma^{-1}(m)} + x_{\sigma^{-1}(m)}.
\end{eqnarray*}
\end{prop}
\begin{proof}
Observe that the shuffling matrix $P^{\sigma}_{n_1,\ldots, n_m}$ defined in \eqref{maindefinition} can be rewritten as
$$
P^{\sigma}_{n_1, n_2, \ldots, n_m} = \sum_{\substack{i_j = 1,\ldots,n_j \\ j=1,\ldots, m}} E^{\sum_{r=1}^{m-1}\left((i_{\sigma^{-1}(r)}-1)\prod_{s=r+1}^mn_{\sigma^{-1}(s)}\right)
+i_{\sigma^{-1}(m)},\sum_{h=1}^{m-1}\left((i_h-1)\prod_{k=h+1}^mn_k\right)+i_m}_{N \times N},
$$
according with Lemma \ref{secondlemma}. Now let
$$
x = \sum_{j=1}^m x_j u_{m-j} = x_1n_2\cdots n_m + x_2 n_3\cdots n_m + \cdots + x_{m-1}n_m + x_m
$$
be the representation of $x$ with respect to the weight vector $[u_{m-1}, u_{m-2}, \ldots, u_0]$. We will determine the image of $x$ by computing the matrix multiplication of the matrix $P^{\sigma}_{n_1, n_2, \ldots, n_m}$ with the column vector ${\bf e}_x^N$ of length $N$, having all entries equal to $0$, except for the entry at position $x$, which is equal to $1$.\\
Notice that this multiplication gives exactly one nonzero contribution, corresponding to the term where the equality $\sum_{h=1}^{m-1}\left((i_h-1)\prod_{k=h+1}^mn_k\right)+i_m =x$ is satisfied. We have only one such term, because of the uniqueness of the representation of the integer $x$ with respect to the weight vector $[u_{m-1}, u_{m-2}, \ldots,u_0]$. Therefore, such a term corresponds to the indices
$$
i_1-1 = x_1; \ i_2-1=x_2; \ldots\ldots ; i_{m-1}-1=x_{m-1}; i_m = x_m.
$$
Consequently, if $\sigma^{-1}(r) = \ell_r \in \{1,\ldots, m\}$, then it must be $i_{\sigma^{-1}(r)} = i_{\ell_r} = x_{\ell_r}+1$, so that the matrix multiplication returns a column vector of length $N$, having all entries equal to $0$, except for the entry at position
$$
x_{\sigma^{-1}(1)} n_{\sigma^{-1}(2)} \cdots n_{\sigma^{-1}(m)} + x_{\sigma^{-1}(2)} n_{\sigma^{-1}(3)} \cdots n_{\sigma^{-1}(m)} +\cdots + x_{\sigma^{-1}(m-1)}n_{\sigma^{-1}(m)} + x_{\sigma^{-1}(m)}.
$$
This gives the assertion.
\end{proof}

In other words, in the new configuration of the $N$ elements under the action of $\widetilde{\sigma}$, we have that the integers from $0$ to $N-1$ are rearranged according with a lexicographic order indexed by the branch indices $(n_{\sigma^{-1}(1)},n_{\sigma^{-1}(2)},\ldots, n_{\sigma^{-1}(m)})$, so that now the $\sigma^{-1}(1)$-th coordinate is the most important, and the $\sigma^{-1}(m)$-th coordinate becomes the less important.

\begin{remark}\rm
The action of the permutations associated with our shuffling permutation matrices coincides with that of the permutations defined by Ronse in \cite{ronse}. Note that the choice $\sigma = (1\ 2\ \ldots \ m)$ returns the perfect shuffle on $N$ elements described in Proposition \ref{proposizioneemme} (see also Remark \ref{remarkdanielone} below). Moreover, the choice $\sigma =(1\ m\ m\!-\!1\ \ldots\ 2)$ returns the perfect shuffle studied by Davio in \cite{davio}. In this last case, in order to determine the image of an element $x=x_1x_2\ldots x_{m-1}x_m$ under the action of $\widetilde{\sigma}$, it suffices to perform a cyclic permutations of the coordinates, from the right to the left, by obtaining $\widetilde{\sigma}(x) = x_2x_3\ldots x_m x_1$ with respect to the new basis vector $[n_2,n_3\ldots, n_m,n_1]$ (see Theorem 4 in \cite{davio}).
\end{remark}

\begin{example}\label{altriptifissi}\rm
Consider the rooted tree $T_{2,2,3}$ of depth $3$ in Figure \ref{FigureExamplepippino} of Example \ref{esempionumerato}. The integers from $0$ to $11$ (in the second row) are listed in the corresponding lexicographic order (first row), that is, the mixed radix representation with respect to the weight vector $[6,3,1]$, associated with the basis vector $[2,2,3]$. The permutation $\widetilde{\sigma}$ rearranges the elements in the lexicographic order given by the mixed radix representation with respect to the weight vector $[4,2,1]$, associated with the basis vector $[3,2,2]$ (fourth row). In the third row, we have the horizontal representation of the permutation $\widetilde{\sigma}$. 
\end{example}

\begin{remark}\label{remarkdanielone}\rm
In the tree representation that we adopt, the permutation on the $N$ elements of the $m$-th level of $T_{n_1, n_2, \ldots, n_m}$ induced by $\sigma=(1\ 2\ \ldots\ m)$ has the following geometrical interpretation. Consider the $N'=\prod_{i=1}^{m-1} n_i$ subtrees of depth $1$ and size $n_m$ constituting the last level of the tree, and rooted at the vertices of the $(m-1)$-st level of $T_{n_1, n_2, \ldots, n_m}$. If we identify such vertices with the set of words $\{x_1\cdots x_{m-1}:\ x_j\in X_j\}$, these subtrees can be naturally denoted by $T_{x_1\ldots x_{m-1}}$. Then the new order given by the permutation $\widetilde{\sigma}$ associated with the shuffling matrix $P^{(1\ 2\ \ldots\ m)}_{n_1,\ldots, n_m}$ is obtained by starting from $0^m$, which is the first vertex of $T_{0^{m-1}}$, corresponding to the integer $0$ in the lexicographic order, and then performing jumps of size $n_m$. In this way, we are consecutively reaching the first vertex of the second subtree $T_{0^{m-2}1}$, then the first vertex of the third subtree $T_{0^{m-2}2}$ and so on until we reach the first vertex of the last subtree $T_{n_{1}\!-\!1\ldots n_{m}\!-\!1}$. After that, we pass to the second vertex of the first subtree $T_{0^{m-1}}$, then to the second vertex of the second subtree $T_{0^{m-1}1}$ and so on by using the same method until we reach the last vertex of the last subtree. See, for instance, Figure \ref{FigureExample}, where we have $n_1=n_2=2$, $n_3=3$ and $\sigma=(1\ 2\ 3)$. This geometrical interpretation can be extended to the permutation associated with any $\sigma\in Sym(m)$, by considering subtrees rooted at higher levels.
\end{remark}

One can ask whether the action of the permutations of $Sym (N)$ induced by the permutations $\sigma \in Sym(m)$ have fixed points, that is, if there exists some integer $x\in \{0,1,\ldots, N-1\}$ satisfying $ \widetilde{\sigma}(x) = x$, for every $\sigma \in Sym(m)$.

\begin{prop}
Let $T_{n_1,\ldots, n_m}$ be the rooted tree with branch indices $(n_1,\ldots, n_m)$. Then, for every $\sigma\in Sym (m)$, the integers $\overline{x_0}=0$ and $\overline{x_{N-1}}=N-1$ are fixed points. Moreover, in the homogeneous case where $n_i= n$ for each $i=1,\ldots, m$, there are exactly $n$ points which are fixed by any permutation $\sigma\in Sym (m)$: they are given by the constant expansions $\overline{x_0} = \underbrace{0\ldots 0}_{m\ times}$, $\overline{x_1}= \underbrace{1\ldots 1}_{m\ times}$, $\ldots$, $\overline{x_{n-1}} = \underbrace{n\!-\!1\ldots n\!-\!1}_{m\ times}$.
\end{prop}
\begin{proof}
Order the integers from $0$ to $N-1$ in the lexicographic order induced by the mixed radix representation with respect to the weight vector $[u_{m-1}, u_{m-2}, \ldots,  u_0]$. We know that the image of the vertex of the $m$-th level of $T_{n_1,\ldots, n_m}$, which is identified with the integer
$$
x = x_1n_2\cdots n_m + x_2 n_3\cdots n_m + \cdots + x_{m-1}n_m + x_m
$$
is
$$
\widetilde{\sigma}(x) = x_{\sigma^{-1}(1)} n_{\sigma^{-1}(2)} \cdots n_{\sigma^{-1}(m)} + x_{\sigma^{-1}(2)} n_{\sigma^{-1}(3)} \cdots n_{\sigma^{-1}(m)} +\cdots + x_{\sigma^{-1}(m-1)}n_{\sigma^{-1}(m)} + x_{\sigma^{-1}(m)}.
$$
In particular, for the integer $0$, we have $x_i = 0$ for each $i=1,\ldots,m$, so that $\widetilde{\sigma}(0)=0$; similarly, for the integer $N-1$, we have $x_i = n_i-1$ for each $i=1,\ldots,m$, so that $\widetilde{\sigma}(N-1)=N-1$.\\ In the homogeneous case, we have
$$
x = \sum_{i=1}^m x_i n^{m-i} \qquad \widetilde{\sigma}(x) = \sum_{i=1}^m x_{\sigma^{-1}(i)} n^{m-i}.
$$
Therefore, for every $h=0,\ldots, n-1$, the integer $x = h\sum_{i=0}^{m-1}n^i = h\frac{n^m-1}{n-1} = h\frac{N-1}{n-1}$ corresponding to the constant expansions $x_i = h$ for each $i=1,\ldots,m$, is a fixed point for every $\sigma\in Sym (m)$.
\end{proof}

\begin{remark}\rm
Observe that, in general, there may be nontrivial fixed points even in the non homogeneous case. For instance, if we choose $n_1=n_3=2, n_2=3$ and $\sigma =(1\ 3)\in Sym(3)$ as in Remark \ref{remarknuovissimo}, then the integers $2,4,7,9$ are fixed by the permutation $\widetilde{\sigma}$ associated with the matrix $P^{(13)}_{2,3,2}$.
\end{remark}

\subsection{The conjugacy property}\label{subsectionconjugacy}
In this section, we provide a generalization of Theorem 6 of \cite{davio} and Proposition 1 of \cite{rose}, showing that the shuffling matrices that we have defined in Section \ref{sectionmaindefinition} are able to rearrange the factors of an iterated Kronecker product of $m$ matrices, according with any new order induced by a permutation $\sigma\in Sym(m)$.

\begin{theorem}\label{theoremconiugio}
For each $i=0,1,\ldots, m-1$, let $A_i = (a_i(p,q))_{\substack{p = 0,\ldots,r_i-1 \\ q=0,\ldots, c_i-1}}\in \mathcal{M}_{r_i \times c_i}(\mathbb{R})$ and let $\sigma\in Sym(m)$. Put
$$
L = \sum_{\substack{i_j = 0,\ldots,r_j-1 \\ j=0,\ldots, m-1}} E^{i_{\sigma^{-1}(0)},i_0}_{r_{\sigma^{-1}(0)}\times r_0} \otimes  E^{i_{\sigma^{-1}(1)},i_1}_{r_{\sigma^{-1}(1)}\times r_1} \otimes \cdots \otimes E^{i_{\sigma^{-1}(m-1)},i_{m-1}}_{r_{\sigma^{-1}(m-1)}\times r_{m-1}}
$$
and
$$
R = \sum_{\substack{j_h = 0,\ldots,c_{h}-1 \\ h=0,\ldots, m-1}} E^{j_{0},j_{\sigma^{-1}(0)}}_{c_{0}\times c_{\sigma^{-1}(0)}} \otimes  E^{j_{1},j_{\sigma^{-1}(1)}}_{c_{1}\times c_{\sigma^{-1}(1)}} \otimes \cdots \otimes E^{j_{m-1},j_{\sigma^{-1}(m-1)}}_{c_{m-1}\times c_{\sigma^{-1}(m-1)}}.
$$
Then
\begin{eqnarray}\label{proofconiugio}
L\cdot (A_0\otimes A_1 \otimes \cdots \otimes A_{m-1})\cdot R = A_{\sigma^{-1}(0)}\otimes A_{\sigma^{-1}(1)} \otimes \cdots \otimes A_{\sigma^{-1}(m-1)}.
\end{eqnarray}
\end{theorem}
\begin{proof}
We will prove our claim by explicitly computing the entry $(p,q)$ of the matrices in both members of \eqref{proofconiugio}. Let
\begin{eqnarray*}\label{expansionpp}
p &=& \sum_{\ell=0}^{m-2}p_{\ell}r_{\sigma^{-1}(\ell+1)}\cdots r_{\sigma^{-1}(m-1)} + p_{m-1}\\
&=& p_0 r_{\sigma^{-1}(1)}r_{\sigma^{-1}(2)}\cdots r_{\sigma^{-1}(m-1)} + \cdots +p_{m-3}r_{\sigma^{-1}(m-2)}r_{\sigma^{-1}(m-1)}+p_{m-2}r_{\sigma^{-1}(m-1)}+p_{m-1}
\end{eqnarray*}
and
\begin{eqnarray*}\label{expansionpp}
q &=& \sum_{\ell=0}^{m-2}q_{\ell}c_{\sigma^{-1}(\ell+1)}\cdots c_{\sigma^{-1}(m-1)} + q_{m-1}\\
&=& q_0 c_{\sigma^{-1}(1)}c_{\sigma^{-1}(2)}\cdots c_{\sigma^{-1}(m-1)} + \cdots +q_{m-3}c_{\sigma^{-1}(m-2)}c_{\sigma^{-1}(m-1)}+q_{m-2}c_{\sigma^{-1}(m-1)}+q_{m-1}
\end{eqnarray*}
be the mixed radix representation of the integers $p$ and $q$ with respect to the basis vectors $[r_{\sigma^{-1}(0)}, r_{\sigma^{-1}(1)}, \ldots, r_{\sigma^{-1}(m-1)}]$ and $[c_{\sigma^{-1}(0)}, c_{\sigma^{-1}(1)}, \ldots, c_{\sigma^{-1}(m-1)}]$, respectively.
Then it follows from the definition of Kronecker product that the entry $(p,q)$ of the matrix on the right-hand member of \eqref{proofconiugio} is equal to
\begin{eqnarray}\label{aula3}
a_{\sigma^{-1}(0)}(p_0,q_0)\cdot a_{\sigma^{-1}(1)}(p_1,q_1)\cdots a_{\sigma^{-1}(m-1)}(p_{m-1},q_{m-1}).
\end{eqnarray}
On the other hand, for the matrix on the left-hand member of \eqref{proofconiugio} the entry $(p,q)$ is given by
$$
\sum_{k=0}^{\prod_{i=0}^{m-1}r_i - 1} \sum_{z=0}^{\prod_{i=0}^{m-1}c_i - 1} L(p,k) A(k,z) R(z,q),
$$
by definition of standard matrix multiplication. Here, we put $A=A_0\otimes A_1\otimes \cdots \otimes A_{m-1}$. Notice that there exists a unique value of $k$ such that the entry $L(p,k)$ is nonzero. By definition of the matrix $L$, such a value $\widetilde{k}$ is equal to
$$
\widetilde{k} = p_{\sigma(0)}r_1\cdots r_{m-1} +  p_{\sigma(1)}r_2\cdots r_{m-1}+ \cdots + p_{\sigma(m-2)} r_{m-1} + p_{\sigma(m-1)}.
$$
Similarly, it follows from the definition of the matrix $R$ that the unique value of $z$ such that the entry $R(z,q)$ is nonzero  is given by
$$
\widetilde{z} = q_{\sigma(0)}c_1\cdots c_{m-1} +  q_{\sigma(1)}c_2\cdots c_{m-1}+ \cdots + q_{\sigma(m-2)} c_{m-1} + q_{\sigma(m-1)}.
$$
More specifically, we have $L(p,\widetilde{k}) = R(\widetilde{z},q) = 1$, so that the entry $(p,q)$ of the matrix on the left-hand member of \eqref{proofconiugio} reduces to
\begin{eqnarray}\label{aula32}
A(\widetilde{k},\widetilde{z}) = a_{0}(p_{\sigma(0)},q_{\sigma(0)})\cdot a_{1}(p_{\sigma(1)},q_{\sigma(1)})\cdots a_{m-1}(p_{\sigma(m-1)},q_{\sigma(m-1)}).
\end{eqnarray}
The claim follows, if we observe that the expressions \eqref{aula3} and \eqref{aula32} coincide.
\end{proof}

In the square case, we get the following conjugacy result.

\begin{theorem}\label{coro1coniugio}
For each $i=0,1,\ldots, m-1$, let $A_i \in \mathcal{M}_{n_i \times n_i}(\mathbb{R})$, and let $\sigma\in Sym(m)$. Put
$$
L = \sum_{\substack{i_j = 0,\ldots,n_j-1 \\ j=0,\ldots, m-1}} E^{i_{\sigma^{-1}(0)},i_0}_{n_{\sigma^{-1}(0)}\times n_0} \otimes  E^{i_{\sigma^{-1}(1)},i_1}_{n_{\sigma^{-1}(1)}\times n_1} \otimes \cdots \otimes E^{i_{\sigma^{-1}(m-1)},i_{m-1}}_{n_{\sigma^{-1}(m-1)}\times n_{m-1}}.
$$
Then
\begin{eqnarray*}
L\cdot (A_0\otimes A_1 \otimes \cdots \otimes A_{m-1}) \cdot L^{-1} = A_{\sigma^{-1}(0)}\otimes A_{\sigma^{-1}(1)} \otimes \cdots \otimes A_{\sigma^{-1}(m-1)}.
\end{eqnarray*}
\end{theorem}
\begin{proof}
It suffices to take into account that any permutation matrix $P$ is an orthogonal matrix, that is, its inverse coincides with its transpose. In formulae, we have $P^T = P^{-1}$. Now we can observe that, if each matrix $A_i$ is a square matrix, then the matrices $L$ and $R$ in Theorem \ref{theoremconiugio} satisfy the equality $R=L^T=L^{-1}$, because of the linearity of the transposition and of the properties of the Kronecker product. The claim follows.
\end{proof}

In the case $m=2$, when $\sigma$ is the nontrivial permutation of the group $Sym(2)$, we get the following results, recovering Theorem 6 in \cite{davio} and Proposition 1 in \cite{rose}.

\begin{corollary}
Let $A_0\in \mathcal{M}_{r_0\times c_0}(\mathbb{R})$ and $A_1\in \mathcal{M}_{r_1\times c_1}(\mathbb{R})$. Put
$$
L = \sum_{\substack{i_0 = 0,\ldots,r_0-1 \\ i_1=0,\ldots, r_1-1}} E^{i_1,i_0}_{r_1\times r_0} \otimes  E^{i_0,i_1}_{r_0\times r_1} \qquad R = \sum_{\substack{j_0 = 0,\ldots,c_0-1 \\ j_1=0,\ldots, c_1-1}} E^{j_{0},j_1}_{c_{0}\times c_1} \otimes  E^{j_{1},j_0}_{c_{1}\times c_0}.
$$
Then
$$
L\cdot (A_0\otimes A_1)\cdot R = A_1\otimes A_0.
$$
Moreover, if $A_0$ and $A_1$ are square matrices of size $n_0$ and $n_1$, respectively, we have
$$
C\cdot (A_0\otimes A_1)\cdot C^{-1} = A_1\otimes A_0,
$$
with
$$
C = \sum_{\substack{i_0 = 0,\ldots,n_0-1 \\ i_1=0,\ldots, n_1-1}} E^{i_1,i_0}_{n_1\times n_0} \otimes  E^{i_0,i_1}_{n_0\times n_1}.
$$
\end{corollary}

\section{Groups generated by shuffling permutations}\label{sectiongruppi}

Perfect shuffles naturally appear  in group theory. Suppose we have a deck of $2n$ cards an we want to perfectly shuffle it: first divide the deck in half and then interleave the two parts. If we number the cards from $0$ to $2n-1$, we can choose that the first card on the top, after performing the shuffle, is $0$ or $n$. This gives rise to two different permutations of the set of the $2n$ cards. Diaconis,
Grahm and Kantor studied in \cite{diaconis} the structure of the subgroup of $Sym(2n)$ generated by the two perfect shuffles described above. They gave a complete classification of such groups
depending on $n$. Medvendoff and Morrison \cite{medvedoffmorrison} generalized the problem of determining the subgroup of the symmetric group generated by the shuffles in the case in which the deck of
cards is divided into several equal piles that can be permuted. Golomb \cite{cuttingshuffling} studied the subgroup $K$ of $Sym(2n)$ generated by the shuffle together with the permutation induced by the usual cut of a deck of $2n$ cards (i.e., placing the top card on the bottom), showing that $K=Sym(2n)$.

In what follows, we are interested in a similar problem (see also \cite{ronse} for an analogous investigation).
We consider the subgroup $K_{n_1, n_2, \ldots, n_m}$ of $Sym(N)$ generated by the permutations $\{\widetilde{\sigma}:\, \sigma\in Sym(m)\}$, in formulae:
$$
K_{n_1, n_2, \ldots, n_m}=\langle \widetilde{\sigma}:\, \sigma\in Sym(m) \rangle \leq Sym(N).
$$

Notice that $K_{n_1, n_2, \ldots, n_m}$ is always a proper subgroup of $Sym(N)$, since it fixes $0$ and $N-1$. This also implies:
$$
[Sym(N):K_{n_1, n_2, \ldots, n_m}]\geq N(N-1).
$$
Clearly, the permutations generating the subgroup $K_{n_1, n_2, \ldots, n_m}$ depend on the order of the branch indices $(n_1,\ldots,n_m)$, as shown in Remark \ref{remarknuovissimo}. In particular, for any reorder $(n_{i_1},\ldots,n_{i_m})$ of the indices $(n_1,\ldots,n_m)$, we get a subgroup of $Sym(N)$. Notice that any reorder can be described by using a suitable permutation of $Sym(m)$: i.e., there exists $\sigma\in Sym(m)$ such that
$$
(n_{i_1},\ldots,n_{i_m})=(n_{\sigma^{-1}(1)}, n_{\sigma^{-1}(2)}, \ldots, n_{\sigma^{-1}(m)}).
$$
There are two interesting cases to be treated: the first one in which the branch indices are equal (homogeneous case), the second in which not all the branch indices coincide and we let them to be reordered (non homogeneous case). Surprisingly, in the non homogeneous case it turns out that the generated subgroup depends only on the branch indices and is independent from their order. These results
enables us to give the description of the subgroup $K_{n,\ldots, n}$ in the homogeneous case.

The key ingredient of our claim is the following result about shuffling matrices, that can be seen as a generalization of Theorem 5 in \cite{davio}. \begin{prop}\label{prop:composition}
Let $T_{n_1, n_2, \ldots, n_m}$ be the tree with branch indices $(n_1, n_2, \ldots, n_m)$, and let $\sigma, \tau \in Sym (m)$. Then
$$
P^{\tau}_{n_{\sigma^{-1}(1)}, n_{\sigma^{-1}(2)}, \ldots, n_{\sigma^{-1}(m)}} \cdot P^{\sigma}_{n_1, n_2, \ldots, n_m}  = P^{\tau\sigma}_{n_1, n_2, \ldots, n_m}.
$$
\end{prop}
\begin{proof}
We have
\begin{eqnarray*}
P^{\tau}_{n_{\sigma^{-1}(1)}, n_{\sigma^{-1}(2)}, \ldots, n_{\sigma^{-1}(m)}}\! \cdot\! P^{\sigma}_{n_1, n_2, \ldots, n_m}\!\!  &\!=\!&\! \!\sum_{\substack{i_h = 1,\ldots,n_h \\ h=1,\ldots, m}} E^{i_{\tau^{-1}\sigma^{-1}(1)},i_{\sigma^{-1}(1)}}_{n_{\tau^{-1}\sigma^{-1}(1)}\times n_{\sigma^{-1}(1)}} \otimes \cdots \otimes E^{i_{\tau^{-1}\sigma^{-1}(m)},i_{\sigma^{-1}(m)}}_{n_{\tau^{-1}{\sigma^{-1}(m)}} \times n_{\sigma^{-1}(m)}}\\
\!\!&\!\cdot\!&\!\! \sum_{\substack{i_j = 1,\ldots,n_j \\ j=1,\ldots, m}} E^{i_{\sigma^{-1}(1)},i_1}_{n_{\sigma^{-1}(1)}\times n_1} \otimes \cdots \otimes E^{i_{\sigma^{-1}(m)},i_m}_{n_{\sigma^{-1}(m)}\times n_m}\cdot\\
\!\!&\!=\!&\!\! \sum_{\substack{i_j = 1,\ldots,n_j \\ j=1,\ldots, m}} E^{i_{\tau^{-1}\sigma^{-1}(1)},i_1}_{n_{\tau^{-1}\sigma^{-1}(1)}\times n_1} \otimes \cdots \otimes E^{i_{\tau^{-1}\sigma^{-1}(m)},i_{m}}_{n_{\tau^{-1}{\sigma^{-1}(m)}}\times n_{m}}\\
\!\!&\!=\!&\!\! P^{\tau\sigma}_{n_1, n_2, \ldots, n_m}
\end{eqnarray*}
according with Lemma \ref{lemmaciccino}. The claim follows.
\end{proof}

The previous proposition asserts that, if we reorder the branch indices in the form $(n_{\sigma^{-1}(1)}, n_{\sigma^{-1}(2)}, \ldots, n_{\sigma^{-1}(m)})$, then any permutation $\tau\in Sym(m)$ induces a permutation of the $m$-th level of $T_{n_{\sigma^{-1}(1)}, n_{\sigma^{-1}(2)}, \ldots, n_{\sigma^{-1}(m)}}$ that can be expressed in terms of permutations of the $m$-th level of $T_{n_1, n_2, \ldots, n_m}$. From this analysis we deduce the following corollary.

\begin{corollary}
Let $T_{n_1, n_2, \ldots, n_m}$ be the tree with branch indices $(n_1, n_2, \ldots, n_m)$, and let $\sigma \in Sym (m)$. Then
$$
(P^{\sigma}_{n_1, n_2, \ldots, n_m})^{-1} = (P^{\sigma}_{n_1, n_2, \ldots, n_m})^T = P^{\sigma^{-1}}_{n_{\sigma^{-1}(1)}, n_{\sigma^{-1}(2)}, \ldots, n_{\sigma^{-1}(m)}}.
$$
\end{corollary}
The following theorem holds.

\begin{theorem}\label{teo:gruppo non omogeneo}
For any $\sigma\in Sym(m)$ the subgroups $K_{n_1, n_2, \ldots, n_m}$ and $K_{n_{\sigma^{-1}(1)}, n_{\sigma^{-1}(2)}, \ldots, n_{\sigma^{-1}(m)}}$ of $Sym(N)$ coincide.
\end{theorem}
\begin{proof}
Let $\sigma\in Sym(m)$ be the permutation inducing the reorder of the branch indices $(n_{\sigma^{-1}(1)}, n_{\sigma^{-1}(2)}, \ldots, n_{\sigma^{-1}(m)})$. We have to prove that, given any $\widetilde{\tau}\in K_{n_{\sigma^{-1}(1)}, n_{\sigma^{-1}(2)}, \ldots, n_{\sigma^{-1}(m)}}$, which is the permutation of the vertices of the $m$-th level of $T_{n_{\sigma^{-1}(1)}, n_{\sigma^{-1}(2)}, \ldots, n_{\sigma^{-1}(m)}}$ induced by $\tau \in Sym(m)$, one has $\widetilde{\tau}\in K_{n_1, n_2, \ldots, n_m}$. This shows that all the generators of $K_{n_{\sigma^{-1}(1)}, n_{\sigma^{-1}(2)}, \ldots, n_{\sigma^{-1}(m)}}$ are contained in $K_{n_1, n_2, \ldots, n_m}$, and, since $\sigma\in Sym(m)$ is arbitrary, the groups $K_{n_{\sigma^{-1}(1)}, n_{\sigma^{-1}(2)}, \ldots, n_{\sigma^{-1}(m)}}$ coincide for all $\sigma\in Sym(m)$.

The action of $\widetilde{\tau}$ can be represented by the matrix $P^{\tau}_{n_{\sigma^{-1}(1)}, n_{\sigma^{-1}(2)}, \ldots, n_{\sigma^{-1}(m)}}$. By Proposition \ref{prop:composition}, one has
\begin{eqnarray}\label{preamboliomo}
P^{\tau}_{n_{\sigma^{-1}(1)}, n_{\sigma^{-1}(2)}, \ldots, n_{\sigma^{-1}(m)}} = P^{\tau\sigma}_{n_1, n_2, \ldots, n_m}\cdot (P^{\sigma}_{n_1, n_2, \ldots, n_m})^{-1}.
\end{eqnarray}
The matrices on the right-hand side of Equation \eqref{preamboliomo} correspond to the permutations $\widetilde{\tau\sigma}, \widetilde{\sigma}^{-1}$ of $K_{n_1, n_2, \ldots, n_m}$ associated with $\tau\sigma, \sigma^{-1}\in Sym(m)$. In particular, $\widetilde{\tau}\in K_{n_1, n_2, \ldots, n_m}$. The proof follows from the arbitrariness of $\sigma$ and $\tau$.
\end{proof}

\begin{remark}\label{remark:gruppi}\rm
It is worth mentioning here that, in general, the map associating $\widetilde{\sigma}$ to $\sigma$ is not an homomorphism from $Sym(m)$ to $K_{n_1, n_2, \ldots, n_m}$. This can be explicitly noticed in Example \ref{esempio finale} (for instance, not all permutations induced by transpositions are of order two). As Corollary \ref{corollario:gruppo omogeneo} shows, the non homogeneity of the tree $T_{n_{\sigma^{-1}(1)}, n_{\sigma^{-1}(2)}, \ldots, n_{\sigma^{-1}(m)}}$ produces an obstruction to that. This can be interpreted as follows: from Proposition \ref{prop:composition}, by applying a permutation $\sigma$, we pass from the lexicographic order of the $m$-th level of the tree $T_{n_1, n_2, \ldots, n_m}$ to a new order of the elements of the $m$-th level that must be considered in the tree $T_{n_{\sigma^{-1}(1)}, n_{\sigma^{-1}(2)}, \ldots, n_{\sigma^{-1}(m)}}$. This implies that the homomorphism does not hold because the structure on which we act is changed.
\end{remark}
\begin{example}\label{esempio finale}\rm
Let $m=3$ and $n_1=n_2=2$ and $n_3=3$, so that, we have $N=12$. Then one can check that:
\begin{itemize}
 \item $\widetilde{(1\ 2)}=(0)(1)(2)(3\ 6)(4\ 7)(5\ 8)(9)(10)(11)$;\\
 \item $\widetilde{(1\ 3)}=(0)(1\ 4\ 6)(2\ 8\ 9\ 3)(5\ 10\ 7)(11)$;\\
 \item $\widetilde{(2\ 3)}=(0)(1\ 2\ 4\ 3)(5)(6)(7\ 8\ 10\ 9)(11)$;\\
 \item $\widetilde{(1\ 2\ 3)}=(0)(1\ 4\ 5\ 9\ 3)(2\ 8\ 10\ 7\ 6)(11)$;\\
 \item $\widetilde{(1\ 3\ 2)}=(0)(1\ 2\ 4\ 8\ 5\ 10\ 9\ 7\ 3\ 6)(11)$.
\end{itemize}
In this case, by using GAP we found that $K_{2,2,3}\simeq C_2 \times (((C_2 \times C_2 \times C_2 \times C_2) \rtimes A_5) \rtimes C_2)$ is a subgroup of order 3840 in $Sym(12)$.
\end{example}

On the other hand, we get a precise description in the homogeneous case.

\begin{corollary}\label{corollario:gruppo omogeneo}
The group $K_{n,\ldots, n}$ is isomorphic to $Sym(m)$.
\end{corollary}
\begin{proof}
Let $\phi:Sym(m)\to \mathcal{M}_{N\times N}(\mathbb{R})$ be the map defined by $\phi(\sigma)=P^{\sigma}_{n, \ldots, n}$. Clearly if $\sigma$ is not trivial it induces a permutation matrix $P^{\sigma}_{n, \ldots, n}$ that is not the identity matrix. Proposition \ref{prop:composition} shows that this map is a homomorphism. This implies that $\phi$ is an isomorphism.
The claim follows by observing that the group $\phi(Sym(m))$ is isomorphic to $K_{n,\ldots, n}$.
\end{proof}

\begin{remark}\rm
In the homogeneous case, one can give an easy interpretation of the group $K_{n,\ldots, n}\simeq Sym(m)$. Given a permutation $\sigma\in Sym(m)$, one can easily deduce from the mixed radix representation that the corresponding permutation $\widetilde{\sigma}\in Sym(N)$ acts on the set of words of length $m$ over the alphabet $\{0,1,\ldots, n-1\}$ by permuting their coordinates. More precisely
$$
\widetilde{\sigma}(x_1\ldots x_m)=x_{\sigma^{-1}(1)}\ldots x_{\sigma^{-1}(m)}.
$$
From this fact, it also follows an alternative proof of Corollary \ref{corollario:gruppo omogeneo}. Notice that such an interpretation does not hold in the general case: if we exchange $x_i$ and $x_j$ by the transposition $(i\ j)$ it may happen that $n_i\neq n_j$ and so the word $x_1\ldots x_{i-1}x_jx_{i+1}\ldots x_{j-1}x_ix_{j+1}\ldots x_{m}$ represents no element of the $m$-th level of $T_{n_1, n_2, \ldots, n_m}$.
\end{remark}

\subsection{The case of perfect shuffles}\label{more}
There is another way to describe shuffles that can be considered. Let $N=\prod_{i=1}^m n_i$ and let $\mathbb{Z}_{N-1}$ be the ring of integers modulo $N-1$. Notice that if $N=N'\times n_m$, with $N'=\prod_{i=1}^{m-1} n_i$ the perfect shuffle on $N$ can be obtained by reordering the $N$ elements $\{0,1,\ldots, N-1\}$ in such a way that $0$ is the first, and then we perform the sum with $n_m$ modulo $(N-1)$ (see Proposition \ref{proposizioneemme} and Remark \ref{remarkdanielone}). This means that the new order induced by the shuffle corresponds to the horizontal representation
$$
0,n_m \textrm{ mod } N-1, 2 n_m \textrm{ mod } N-1, \ldots\ldots, N-1.
$$
We denote by $Sh_{n_m}$ such operation. In the setting described in the previous sections, this corresponds to act by the permutation $\widetilde{\sigma}$ induced by $\sigma=(1\ 2\ \ldots\ m)\in Sym(m)$ on the rooted tree $T_{n_1,n_2,\ldots, n_m}$, so that the new order can be regarded as the horizontal representation of the permutation $\widetilde{\sigma}$ (see Proposition \ref{proposizioneemme} and Example \ref{esempioemme}).

For any $k=1,\ldots, N-2$, we can more generally define the operation $Sh_{k}$ in an analogous way, just by performing the sum with $k$ modulo $N-1$. Notice that $Sh_1$ is the identity permutation on the set $\{0,1,\ldots, N-1\}$. We want to study the cases in which the reordering of the $N$ elements gives rise to a permutation of them, which fixes the first and the last elements ($0$ and $N-1$). Notice that this happens if and only if starting from $0$ and considering its orbit under $Sh_k$ we do find $N-1$ distinct elements. The following simple result clarifies when this operation gives rise to a permutation.

\begin{lemma}
The orbit of $0$ under the action of $Sh_k$ contains $N-1$ elements if and only if $(k,N-1)=1$.
\end{lemma}
\begin{proof}
It suffices to consider that the orbit of $0$ under the action of $Sh_k$ contains $N-1$ elements if and only if $k$ is a generator of the additive group $\mathbb{Z}_{N-1}$.
\end{proof}

\begin{remark}\rm
The perfect shuffle of Proposition \ref{proposizioneemme} is obtained by setting $k=n_m$. It is clear, in this case, that $(n_m,N-1)=1$.
\end{remark}

The previous lemma states that only the $Sh_k$'s, with $k$ coprime to $N-1$, give rise to permutations of $N$ elements. In the group theoretical view of the shuffle, one can ask what is the subgroup of $Sym(N)$ generated by the permutations $Sh_k$, where $(k,N-1)=1$. We denote such a group by $G_{Sh, N}$. We have the following description.

\begin{prop}\label{gruppodeglishuffling}
 Let $\mathbb{Z}_{N-1}^{\times}$ be the multiplicative subgroup of the invertible elements in $\mathbb{Z}_{N-1}$. Then
 $$
 G_{Sh, N}\simeq \mathbb{Z}_{N-1}^{\times}
 $$
\end{prop}
\begin{proof}
First notice that $Sh_k$ acts on the ordered set $\{0,1,\ldots, N-2\}$ by performing the multiplication by $k$ modulo $N-1$. In fact, the order induced by $Sh_k$ is given by
$$
0, k \textrm{ mod } N-1,  2k\textrm{ mod } N-1, \ldots, (N-2)k \textrm{ mod } N-1
$$
which is exactly the multiplication by $k$ in $\mathbb{Z}_{N-1}$. Therefore, we can construct a map $\phi: \mathbb{Z}_{N-1}^{\times}\to G_{Sh, N}$ such that $\phi(k)=Sh_k$.
Let us prove that $\phi$ is an isomorphism by showing that it is an injective homomorphism. We have that $\phi(hk)=Sh_{hk}$ is the multiplication by $hk$ modulo $N-1$. On the other hand $\phi(h)\phi(k)$ is the multiplication by $k$ modulo $N-1$, followed by the multiplication by $h$ modulo $N-1$. The equality holds since
$$
hk \cdot x \ {\rm mod} \ (N-1) = h\cdot (kx \ {\rm mod} \ (N-1)) {\rm mod } (N-1).
$$
It is moreover clear that $\phi$ is injective, since the only action fixing everything is that induced by $Sh_1=\phi(1)$.
\end{proof}

As a consequence, we have $|G_{Sh, N}| =|\mathbb{Z}_{N-1}^{\times}| = \varphi(N-1)$, where $\varphi$ denotes the Euler function.

\begin{example}\rm
For $N=8$ (see Table \ref{TableSh}) and $N=15$ (see Table \ref{TableSh2}), the groups $\mathbb{Z}_{7}^{\times}$ and $\mathbb{Z}_{14}^{\times}$ are both isomorphic to the cyclic group of $6$ elements.
\begin{table}\small
\begin{tabular}{|c|c|c|c|c|c|c|c|c|}
\hline
$N=8$  & $0$ & $1$ & $2$ & $3$ & $4$ & $5$ & $6$ & $7$   \\
\hline
$Sh_1$  & $0$ & $1$ & $2$ & $3$ & $4$ & $5$ & $6$ & $7$   \\
\hline
$Sh_2$  & $0$ & $2$ & $4$ & $6$ & $1$ & $3$ & $5$ & $7$   \\
 \hline
$Sh_3$  & $0$ & $3$ & $6$ & $2$ & $5$ & $1$ & $4$ & $7$   \\
 \hline
$Sh_4$  & $0$ & $4$ & $1$ & $5$ & $2$ & $6$ & $3$ & $7$   \\
 \hline
$Sh_5$  & $0$ & $5$ & $3$ & $1$ & $6$ & $4$ & $2$ & $7$   \\
 \hline
$Sh_6$  & $0$ & $6$ & $5$ & $4$ & $3$ & $2$ & $1$ & $7$   \\
 \hline
\end{tabular} \caption{The permutations induced by $Sh_k$ for $N=8$.} \label{TableSh}
\end{table}
\begin{table}\small
\begin{tabular}{|c|c|c|c|c|c|c|c|c|c|c|c|c|c|c|c|}
\hline
$N=15$  & $0$ & $1$ & $2$ & $3$ & $4$ & $5$ & $6$ & $7$ & $8$ & $9$ & $10$ & $11$ & $12$ & $13$ & $14$   \\
\hline
$Sh_1$  & $0$ & $1$ & $2$ & $3$ & $4$ & $5$ & $6$ & $7$ & $8$ & $9$ & $10$ & $11$ & $12$ & $13$ & $14$   \\
\hline
$Sh_3$  & $0$ & $3$ & $6$ & $9$ & $12$ & $1$ & $4$ & $7$ & $10$ & $13$ & $2$ & $5$ & $8$ & $11$ & $14$   \\
 \hline
$Sh_5$  & $0$ & $5$ & $10$ & $1$ & $6$ & $11$ & $2$ & $7$ & $12$ & $3$ & $8$ & $13$ & $4$ & $9$ & $14$   \\
 \hline
$Sh_9$  & $0$ & $9$ & $4$ & $13$ & $8$ & $3$ & $12$ & $7$ & $2$ & $11$ & $6$ & $1$ & $10$ & $5$ & $14$   \\
 \hline
$Sh_{11}$  & $0$ & $11$ & $8$ & $5$ & $2$ & $13$ & $10$ & $7$ & $4$ & $1$ & $12$ & $9$ & $6$ & $3$ & $14$   \\
 \hline
$Sh_{13}$  & $0$ & $13$ & $12$ & $11$ & $10$ & $9$ & $8$ & $7$ & $6$ & $5$ & $4$ & $3$ & $2$ & $1$ & $14$   \\
 \hline
\end{tabular} \caption{The permutations induced by $Sh_k$ for $N=15$.} \label{TableSh2}
\end{table}
\end{example}

\begin{remark}\rm
We address now the question about which permutations of type $Sh_k$ can be obtained by our shuffling matrix construction introduced in Section \ref{mainsection}. If $N=n_1\cdots n_m$, we have defined the shuffles $Sh_k$ where $(k,N-1)=1$. Notice that, each factor of $N$ is coprime to $N-1$ and so it is an element of $\mathbb{Z}_{N-1}^{\times}$.
Suppose that $N=hk$, with $h,k\neq 1$, then by considering an opportune rooted tree $T$ whose last level has branch index $k$ (actually, it is enough to consider the tree $T_{h,k}$ of depth $2$ with branch indices $(h,k)$) we have that the permutation of $Sym(N)$ given by $Sh_k$ coincides with the perfect shuffle induced on the tree by choosing the cyclic permutation $\sigma=(1\ 2\ \ldots)$ whose length is the depth of $T$. On the other hand, if $k$ is not a factor of $N$, but it is coprime to $N-1$, then $Sh_k$ induces a permutation that cannot be represented through a shuffling matrix acting on $N$ elements, as defined in Section \ref{mainsection}.
\end{remark}

\section{An application to the Discrete Fourier Transform}\label{sectionfourier}

The Discrete Fourier Transform (DFT) is extensively used in a large number of fields, going from Pure mathematics (Number theory, Harmonic analysis, Partial Differential Equations) to several applications (Digital signal processing, including speech analysis and radar, or Image processing). The basic setting of finite Fourier analysis is represented by the multiplicative group $\mathbb{Z}(N)$ consisting of the $N$-th roots of the unity. Such group is the natural discrete approximation of the circle, as $N$ goes to infinity, and for this reason it is a very good candidate for the storage of information of any function on the circle, and for all the numerical applications involving Fourier series.\\
One of the most largely used way to compute the DFT is by means of the so called \textit{discrete Fourier transform matrix}, which is a particular example of a \textit{Vandermonde matrix}. Recall that the Vandermonde matrix $V(a_1,\ldots, a_N)$ is defined as the square matrix of size $N$ given by
$$
V(a_1,\ldots, a_N) = \left(
                       \begin{array}{ccccc}
                         1 & 1 & \cdots & \cdots & 1 \\
                         a_1 & a_2 & \cdots & \cdots & a_N \\
                         a_1^2 & a_2^2 & \cdots & \cdots & a_N^2 \\
                         \vdots & \vdots &  &  & \vdots \\
                         a_1^{N-1} & a_2^{N-1} & \cdots & \cdots & a_N^{N-1} \\
                       \end{array}
                     \right).
$$
For each $N\in \mathbb{N}$, the $N\times N$ discrete Fourier transform matrix is defined to be the matrix $F_N(\omega) = (f_{ij})_{i,j=0,\ldots,N-1}$, with
$$
f_{ij} = \omega^{ij}, \quad i,j=0,1,\ldots, N-1; \quad \omega = \exp\left(\frac{2\pi i}{N}\right).
$$
Then it can be easily seen that $F_N(\omega) = V(1,\omega,\omega^2,\ldots, \omega^{N-1})$. The DFT of an $N$-vector $x=(x_0,x_1,\ldots, x_{N-1})^T$ is the vector $y=(y_0,y_1,\ldots, y_{N-1})^T$ defined by the formula
\begin{eqnarray}\label{startfrom}
y_j = \sum_{k=0}^{N-1}x_k\omega^{jk}, \qquad 0\leq j\leq N-1,
\end{eqnarray}
that is, by performing the standard multiplication $y = F_N(\omega)x$.

The problem which naturally arises in numerical analysis, is to determine an algorithm minimizing the amount of time it takes a computer to calculate the Fourier coefficients of a given function on $\mathbb{Z}(N)$. The Fast Fourier Transform (FFT) is a method developed with the aim of calculating efficiently such coefficients.
For this reason, many numerical algorithms implementing the DFT usually employ FFT algorithms.
In particular, the situation is quite interesting in the case $N=2^n$,  since the computations of the Fourier coefficients in this case uses the fact that an induction in $n$ requires only $\log N$ steps to go from $N=1$ to $N=2^n$, saving time in the practical applications. Roughly speaking, an FFT rapidly computes such coefficients by factorizing the DFT matrix into a product of sparse (mostly zero) blocks. Therefore, it manages to reduce the complexity of computing the DFT from $O(N^2)$, which arises if one simply applies the definition of DFT, to $O(N \log N)$, where $N$ is the data size \cite{brigham}. For these reasons, Fast Fourier Transforms are ubiquitously used for applications in several branches of Engineering.

In the seminal paper \cite{seminal} by Cooley and Tukey, the authors start from \eqref{startfrom} and map the integers $j$ and $k$ bijectively to ordered pairs $(p,q)$, with $0\leq p\leq s-1$ and $0\leq q\leq r-1$, with $N=rs$, replacing \eqref{startfrom} by a permuted double summation, showing that the resulting summation can be interpreted as a series of smaller DFT computations.\\ Many variations and generalizations of this construction were obtained in the literature: the best-known FFT algorithms depend on the factorization of $N$, even if there exist FFTs with $O(N \log N)$ complexity also for prime $N$; more general versions of FFT are applicable when $N$ is composite and not necessarily a power of $2$ (see, for instance, the Bluestein algorithm using a convolution approach \cite{segnali}).

In the paper \cite{rose} by Rose, shuffling permutations are applied in the setting of Discrete Fourier Transform. More precisely, Theorem 1 shows the following fundamental identity, called {\it General radix identity}, in the case $n=rs$:
\begin{eqnarray}\label{allabase}
 F_n(\omega) P_s^r = (F_r(\omega^s)\otimes I_s)T^s_r (I_r\otimes F_s(\omega^r)),
\end{eqnarray}
with $T^s_r(\omega) = D_r(D_s(\omega))$, and
$$
D_s(W) = diag(W^0,W^1, \ldots, W^{s-1}) = \left(
                                            \begin{array}{ccccc}
                                              W^0 & 0 & \cdots & \cdots & 0 \\
                                              0 & W^1 &  &  & \vdots \\
                                              \vdots &  & \ddots &  & \vdots \\
                                              \vdots &  &  & \ddots & 0 \\
                                              0 & \cdots & \cdots & 0 & W^{s-1} \\
                                            \end{array}
                                          \right),
$$
where $W^i$ denotes the $i$-th power of the square matrix $W$. The Equation \eqref{allabase} is the basis of the Cooley-Tukey FFT and all its commonly used variants. It is easy to see that, in the case $n = n_1n_2$, with $n_2=2$, the identity above can also be rewritten as
\begin{eqnarray}\label{fourierparticolare}
F_{2n_1}(\omega) (P^{(12)}_{n_1,2})^T &=& \left(
                                     \begin{array}{cc}
                                       F_{n_1}(\omega^2) & D_{n_1}(\omega)F_{n_1}(\omega^2)  \\
                                       F_{n_1}(\omega^2) & -D_{n_1}(\omega)F_{n_1}(\omega^2) \\
                                     \end{array}
                                   \right)\\
                                   &=& \left(
                                     \begin{array}{cc}
                                       I_{n_1} & D_{n_1}(\omega) \\
                                       I_{n_1} & -D_{n_1}(\omega) \\
                                     \end{array}
                                   \right)  \left(
                                     \begin{array}{cc}
                                       F_{n_1}(\omega^2) & 0 \\
                                       0 & F_{n_1}(\omega^2) \\
                                     \end{array}
                                   \right).\nonumber
\end{eqnarray}
The aim of this section is to generalize Equation \eqref{fourierparticolare} to any setting $N = n_1n_2\cdots n_m$, for any permutation $\sigma \in Sym(m)$.
The main result of this section is contained in the following fundamental theorem.

\begin{theorem}\label{generaltheoremfourier}
Let $N=n_1n_2\cdots n_m$, and let $P^{\sigma}_{n_1,n_2,\ldots,n_m}$ be the shuffling matrix defined as in \eqref{maindefinition}, with $\sigma\in Sym(m)$. Let $\widetilde{\sigma}$ be the permutation of the elements $\{0,1\ldots, N-1\}$ induced by $P^{\sigma}_{n_1,\ldots, n_m}$. Put $\omega = \exp\left(\frac{2\pi i}{N}\right)$. Then:
\begin{eqnarray}\label{blocchi3}
F_{N}(\omega)(P^{\sigma}_{n_1,n_2,\ldots,n_m})^T &=& \left(
                                                                 \begin{array}{ccccc}
                                                                   B_{00} & B_{01} & \cdots & \cdots & B_{0\ n_{m}\!-\!1}\\
                                                                   B_{10} & B_{11} &\cdots & \cdots & \vdots \\
                                                                   \vdots &  &  &  & \vdots \\
                                                                   \vdots &  &  &  & \vdots \\
                                                                   B_{n_m\!-\!1\ 0} & B_{n_m\!-\!1\ 1} & \cdots & \cdots & B_{n_m\!-\!1\ n_m\!-\!1}  \\
                                                                 \end{array}
                                                               \right)
\end{eqnarray}
where, for each $h,k=0,\ldots, n_m-1$, one has $B_{hk}=C_{hk}A_h$, with
\begin{eqnarray*}
C_{hk} = \omega^{\frac{Nh}{n_m}\widetilde{\sigma}^{-1}\left(\frac{Nk}{n_m}\right)}\cdot D_{N/n_m}(\omega^{\widetilde{\sigma}^{-1}\left(\frac{Nk}{n_m}\right)})
\end{eqnarray*}
and $A_h$ is the square matrix of size $N/n_m$, whose entry $(i,j)$ is equal to $\omega^{(i+\frac{hN}{n_m})\widetilde{\sigma}^{-1}(j)}$, for each $i,j=0,1,\ldots, \frac{N}{n_m}-1$.
\end{theorem}
\begin{proof}
We will prove our claim by showing that the entries $(i,j)$ of the matrices on both the members of \eqref{blocchi3} coincide.
Let
$$
i = i_mn_1\cdots n_{m-1}+ i_1n_2\cdots n_{m-1}+ \cdots + i_{m-2}n_{m-1}+i_{m-1}
$$
and
$$
j = j_mn_1\cdots n_{m-1}+ j_1n_2\cdots n_{m-1}+ \cdots + j_{m-2}n_{m-1}+j_{m-1}
$$
be the representations of the integers $i$ and $j$, respectively, with respect to the vector basis $[n_m, n_{1}, n_2, \ldots, n_{m-1}]$. This implies that the entry $(i,j)$ of the matrix on the right-hand member of \eqref{blocchi3} belongs to the block $B_{i_m j_m}=C_{i_mj_m}A_{i_m}$, since each square block $B_{hk}$ has size $n_1\cdots n_{m-1}$ by definition. In this case, we have $\frac{Nk}{n_m} = j_mn_1\ldots n_{m-1}$, so that $\widetilde{\sigma}^{-1}(Nk/n_m) = j_{\sigma(m)}n_{\sigma(1)}\cdots n_{\sigma(m-1)}$. Moreover, the term of the matrix $D_{N/n_m}(\omega^{\widetilde{\sigma}^{-1}\left(\frac{Nj_m}{n_m}\right)})$ contributing to the entry $(i,j)$ of the matrix on the right-hand member of Equation \eqref{blocchi3} is its $(i_1n_2\cdots n_{m-1}+ \cdots + i_{m-2}n_{m-1}+i_{m-1})$-th diagonal entry, given by
$$
\omega^{(i_1n_2\cdots n_{m-1}+ \cdots + i_{m-2}n_{m-1}+i_{m-1})j_{\sigma(m)}n_{\sigma(1)}\cdots n_{\sigma(m-1)}}.
$$
Similarly, the entry of $A_{i_m}$ contributing to $(i,j)$ is its entry at row $ i_1n_2\cdots n_{m-1}+ \cdots + i_{m-2}n_{m-1}+i_{m-1}$ and column $j_1n_2\cdots n_{m-1}+ \cdots + j_{m-2}n_{m-1}+j_{m-1}$, which is given by
$$
\omega^{((i_1n_2\cdots n_{m-1}+ \cdots + i_{m-2}n_{m-1}+i_{m-1})+i_mn_1\cdots n_{m-1})\cdot (j_{\sigma(1)}n_{\sigma(2)}\cdots n_{\sigma(m-1)}+ \cdots + j_{\sigma(m-2)}n_{\sigma(m-1)}+j_{\sigma(m-1)})}
$$
$$
= \omega^{i(j_{\sigma(1)}n_{\sigma(2)}\cdots n_{\sigma(m-1)}+ \cdots + j_{\sigma(m-2)}n_{\sigma(m-1)}+j_{\sigma(m-1)})}.
$$
By taking the product of all contributions, we get the following value:
$$
\omega^{i_mn_1\ldots n_{m-1}j_{\sigma(m)}n_{\sigma(1)}\cdots n_{\sigma(m-1)}}\cdot \omega^{(i_1n_2\cdots n_{m-1}+ \cdots + i_{m-2}n_{m-1}+i_{m-1})j_{\sigma(m)}n_{\sigma(1)}\cdots n_{\sigma(m-1)}}\cdot
$$
$$
\cdot \omega^{i(j_{\sigma(1)}n_{\sigma(2)}\cdots n_{\sigma(m-1)}+ \cdots + j_{\sigma(m-2)}n_{\sigma(m-1)}+j_{\sigma(m-1)})}
$$
$$
= \omega^{i(j_{\sigma(m)}n_{\sigma(1)}\cdots n_{\sigma(m-1)}+j_{\sigma(1)}n_{\sigma(2)}\cdots n_{\sigma(m-1)}+ \cdots + j_{\sigma(m-2)}n_{\sigma(m-1)}+j_{\sigma(m-1)})}
$$
Let us consider now the entry $(i,j)$ of the matrix on the left-hand member of Equation \eqref{blocchi3}. From the definition of matrix multiplication, such entry is equal to
$$
\sum_{\ell=0}^{N-1} \omega^{i\ell} p_{j\ell},
$$
where $p_{j\ell}$ denotes the entry $(j,\ell)$ of the matrix $P^{\sigma}_{n_1,n_2,\ldots,n_m}$. Now it is clear that there exists a unique index $\ell$ such that the entry $p_{j\ell}$ is nonzero. More precisely, the entry $p_{j\ell}$ is nonzero and equals $1$ if and only if
$$
\ell = j_{\sigma(m)}n_{\sigma(1)}\cdots n_{\sigma(m-1)}+j_{\sigma(1)}n_{\sigma(2)}\cdots n_{\sigma(m-1)}+ \cdots + j_{\sigma(m-2)}n_{\sigma(m-1)}+j_{\sigma(m-1)}.
$$
This ensures that the entry $(i,j)$ of the left-hand member of \eqref{blocchi3} equals
$$
\omega^{i(j_{\sigma(m)}n_{\sigma(1)}\cdots n_{\sigma(m-1)}+j_{\sigma(1)}n_{\sigma(2)}\cdots n_{\sigma(m-1)}+ \cdots + j_{\sigma(m-2)}n_{\sigma(m-1)}+j_{\sigma(m-1)})},
$$
and this completes the proof.
\end{proof}

In the case $\sigma = (1\ 2\ \ldots\ m)\in Sym(m)$, we get the following result.

\begin{prop}\label{propositionparticolare}
Let $N=n_1n_2\cdots n_m$, and consider the permutation matrix $P^{(1\ 2\ \ldots m)}_{n_1,n_2,\ldots,n_m}$. Put $\omega = \exp\left(\frac{2\pi i}{N}\right)$. Then:
\begin{eqnarray}\label{blocchi}
F_{N}(\omega)(P^{(1\ 2\ \ldots m)}_{n_1,n_2,\ldots,n_m})^T &=& \left(
                                                                 \begin{array}{ccccc}
                                                                   B_{00} & B_{01} & \cdots & \cdots & B_{0\ n_{m}\!-\!1}\\
                                                                   B_{10} & B_{11} &\cdots & \cdots & \vdots \\
                                                                   \vdots &  &  &  & \vdots \\
                                                                   \vdots &  &  &  & \vdots \\
                                                                   B_{n_m\!-\!1\ 0} & B_{n_m\!-\!1\ 1} & \cdots & \cdots & B_{n_m\!-\!1\ n_m\!-\!1}  \\
                                                                 \end{array}
                                                               \right)
\end{eqnarray}
where $B_{hk} = \omega^{\frac{Nhk}{n_m}}\cdot D_{N/n_m}(\omega^k)\cdot F_{N/n_m}(\omega^{n_m})$, for every $h,k=0,\ldots, n_m-1$.
\end{prop}
\begin{proof}
It suffices to observe that, in the case $\sigma=(1\ 2 \ldots\ m)$, we have $\widetilde{\sigma}^{-1}\left(\frac{Nk}{n_m}\right)=k$, and that the entry $(i,j)$ of $A_h$ reduces to $\omega^{(i+\frac{hN}{n_m})jn_m}=\omega^{ijn_m}$, since $\widetilde{\sigma}^{-1}(j) = jn_m$ in this case. Therefore, the matrix $A_h$ reduces to the matrix $F_{N/n_m}(\omega^{n_m})$, for every $h$.
\end{proof}

\begin{remark}\rm
Note that, if we put $k=0$ in Theorem \ref{generaltheoremfourier}, then we get $B_{h0} = A_h$, for every $h=0,\ldots, n_m-1$. Similarly, if we put $k=0$ in Proposition \ref{propositionparticolare}, we get $B_{h0} = F_{N/n_m}(\omega^{n_m})$, for each $h=0,\ldots, n_m-1$.

Moreover, the matrix in Equation \eqref{blocchi} can be also reformulated as follows:
\begin{eqnarray*}\label{blocchi2}
\left(
                                                                 \begin{array}{cccc}
                                                                   C_{00} & C_{01} & \cdots & C_{0\ n_{m}\!-\!1}\\
                                                                   C_{10} & C_{11} & \cdots & \vdots \\
                                                                   \vdots &  &   & \vdots \\
                                                                  C_{n_m\!-\!1\ 0} & C_{n_m\!-\!1\ 1} & \cdots & C_{n_m\!-\!1\ n_m\!-\!1}\\
                                                                   \end{array}
                                                               \right)\cdot
                                                               \left(
                                                                                                                     \begin{array}{cccc}
  F_{N/n_m}(\omega^{n_m}) & 0  & \cdots & 0 \\
                                                                                   0 & F_{N/n_m}(\omega^{n_m}) &    & \vdots \\

                                                                                  \vdots &  &  \ddots & 0 \\
                                                                                   0 & \cdots & 0 & F_{N/n_m}(\omega^{n_m}) \\
                                                                                 \end{array}
                                                                               \right)
\end{eqnarray*}
with $C_{hk} = \omega^{\frac{Nhk}{n_m}}D_{N/n_m}(\omega^k)$. Notice that, for every $h=0,\ldots, n_m-1$, we have $C_{h0} = I_{n_m}$. In particular, for $m=2$, $n_2=2$ and $\sigma=(1\ 2)$, we obtain Equation \eqref{fourierparticolare}.
\end{remark}

\begin{example}\rm
Let $n_1=n_2=2$ and $n_3=3$, so that $N=12$. Choose the permutation $\sigma=(1\ 3) \in Sym(3)$. Then we have
\begin{eqnarray*}
P^{(1\ 3)}_{2,2,3} &=& \sum_{\substack{i_j = 1,\ldots,n_j \\ j=1,2,3}} E^{i_3,i_1}_{3\times 2} \otimes E^{i_2,i_2}_{2\times 2} \otimes E^{i_1,i_3}_{2\times 3}\\
&=&   E^{1,1}_{3\times 2} \otimes I_2 \otimes E^{1,1}_{2\times 3}+ E^{2,1}_{3\times 2} \otimes I_2 \otimes E^{1,2}_{2\times 3}+ E^{3,1}_{3\times 2} \otimes I_2 \otimes E^{1,3}_{2\times 3}\\
 &+&  E^{1,2}_{3\times 2} \otimes I_2 \otimes E^{2,1}_{2\times 3}+ E^{2,2}_{3\times 2} \otimes I_2 \otimes E^{2,2}_{2\times 3}+ E^{3,2}_{3\times 2} \otimes I_2 \otimes E^{2,3}_{2\times 3}.
\end{eqnarray*}
For $\omega = \exp\left(\frac{2\pi i}{12}\right)$, let $F_{12}(\omega) = (\omega^{ij})_{i,j=0,\ldots,11}$ be the Discrete Fourier Matrix of size $12$. Then it holds:
\begin{eqnarray*}
F_{12}(\omega)(P^{(1\ 3)}_{2,2,3})^T &=& \left(
                                       \begin{array}{ccc}
                                         B_{00} & B_{01} & B_{02} \\
                                         B_{10} & B_{11} & B_{12} \\
                                         B_{20} & B_{21} & B_{22} \\
                                       \end{array}
                                     \right)\\
&=&\left(\begin{array}{cccc|cccc|cccc}
                                        1 & 1 & 1 & 1 & 1 & 1 & 1 & 1 & 1 & 1 & 1 & 1 \\
                                        1 & \omega^6 & \omega^3 & \omega^9 & \omega & \omega^7 & \omega^4 & \omega^{10} & \omega^2 & \omega^8 & \omega^5 & \omega^{11} \\
                                        1 & 1 & \omega^6 & \omega^6 & \omega^2 & \omega^2 & \omega^8 & \omega^8 & \omega^4 & \omega^4 & \omega^{10} & \omega^{10} \\
                                        1 & \omega^6 & \omega^9 & \omega^3 & \omega^3 & \omega^9 & 1 & \omega^6 & \omega^6 & 1 & \omega^3 & \omega^9 \\  \hline
                                        1 & 1 & 1 & 1 & \omega^4 & \omega^4 & \omega^4 & \omega^4 & \omega^8 & \omega^8 & \omega^8 & \omega^8 \\
                                        1 & \omega^6 & \omega^3 & \omega^9 & \omega^5 & \omega^{11} & \omega^8 & \omega^2 & \omega^{10} & \omega^4 & \omega & \omega^7 \\
                                        1 & 1 & \omega^6 & \omega^6 & \omega^6 & \omega^6 & 1 & 1 & 1 & 1 & \omega^6 & \omega^6 \\
                                        1 & \omega^6 & \omega^9 & \omega^3 & \omega^7 & \omega & \omega^4 & \omega^{10} & \omega^2 & \omega^8 & \omega^{11} & \omega^5 \\   \hline
                                        1 & 1 & 1 & 1 & \omega^8 & \omega^8 & \omega^8 & \omega^8 & \omega^4 & \omega^4 & \omega^4 & \omega^4 \\
                                        1 & \omega^6 & \omega^3 & \omega^9 & \omega^9 & \omega^3 & 1 & \omega^6 & \omega^6 & 1 & \omega^9 & \omega^3 \\
                                        1 & 1 & \omega^6 & \omega^6 & \omega^{10} & \omega^{10} & \omega^4 & \omega^4 & \omega^8 & \omega^8 & \omega^2 & \omega^2 \\
                                        1 & \omega^6 & \omega^9 & \omega^3 & \omega^{11} & \omega^5 & \omega^8 & \omega^2 & \omega^{10} & \omega^4 & \omega^7 & \omega \\
                                      \end{array}
                                    \right).
\end{eqnarray*}
Here, we have $B_{hk} = C_{hk}A_h$, for all $h,k=0,1,2$, where
$$
C_{hk} = \omega^{4h\widetilde{\sigma}^{-1}(4k)}  \cdot D_4(\omega^{\widetilde{\sigma}^{-1}(4k)}),
$$
and
$$
A_0(i,j) = \omega^{i\widetilde{\sigma}^{-1}(j)}  \qquad  A_1(i,j) = \omega^{(i+4)\widetilde{\sigma}^{-1}(j)}    \qquad A_2(i,j) = \omega^{(i+8)\widetilde{\sigma}^{-1}(j)}
$$
for $i,j=0,1,2,3$, with $\widetilde{\sigma}^{-1}(0)=0$; $\widetilde{\sigma}^{-1}(1)=6$; $\widetilde{\sigma}^{-1}(2)=3$; $\widetilde{\sigma}^{-1}(3)=9$, since $\widetilde{\sigma} = (0)(1\ 4\ 6)(2\ 8\ 9\ 3)(5\ 10\ 7)(11)$. Observe that, in this special example, we get $A_0 = A_1 = A_2$, what is not true in general.
\end{example}

\section*{Acknowledgments}
We want to express our deepest gratitude to Fabio Scarabotti for enlightening discussions and for his continuous encouragement. 


\begin{thebibliography}{99}

\bibitem{aldousdiaconis} D. Aldous, P. Diaconis, Shuffling cards and stopping times,  {\it Amer. Math. Monthly} {\bf 93} (1986), no. 5, 333--348. \texttt{doi:10.2307/2323590}
 
\bibitem{bayerdiaconis} D. Bayer, P. Diaconis, Trailing the dovetail shuffle to its lair, {\it Ann. Appl. Probab.} {\bf 2} (1992), no. 2, 294--313. \texttt{doi:10.1214/aoap/1177005705}

\bibitem{brigham} E.O. Brigham, The Fast Fourier Transform, Prentice-Hall Inc, Englewood Cliffs, New Jersey, 1974.

\bibitem{seminal} J.W. Cooley, J.W. Tuckey, An algorithm for the machine calculation of complex Fourier series, {\it Math. Comp.} {\bf 19} (1965), 297--301. \texttt{doi:10.1090/S0025-5718-1965-0178586-1}

\bibitem{dado} D. D'Angeli, A. Donno, No cut-off phenomenon for the ``insect Markov chain'', {\it Monatsh. Math.} {\bf 156} (2009), no. 3, 201--210. \texttt{doi:10.1007/s00605-008-0014-x}

\bibitem{dado2} D. D'Angeli, A. Donno, Crested products of Markov chains, {\it Ann. Appl. Probab.} {\bf 19} (2009), no. 1, 414--453. \texttt{doi:10.1214/08-AAP546}

\bibitem{dado3} D. D'Angeli, A. Donno, Markov chains on orthogonal block structures, {\it European J. Combin.} {\bf 31} (2010), no. 1, 34--46. \texttt{doi:10.1016/j.ejc.2009.02.003}

\bibitem{davio} M. Davio, Kronecker Products and Shuffle Algebra, {\it IEEE Trans. Comp.} {\bf 30} (1981), no. 2, 116--125. \texttt{doi:10.1109/TC.1981.6312174}

\bibitem{diaconis} P. Diaconis, R.L. Graham, W.M. Kantor, The mathematics of perfect shuffles, {\it Adv. in Appl. Math.} {\bf 4} (1983), no. 2, 175--196. \texttt{doi:10.1016/0196-8858(83)90009-X}

\bibitem{diaconis2} P. Diaconis, Mathematical developments from the analysis of riffle shuffling, {\it Groups, Combinatorics \& Geometry (Durham, 2001)}, 73--97, World Sci. Publ., River Edge, NJ, 2003.

\bibitem{diaconis3} P. Diaconis, The cutoff phenomenon in finite Markov chains,
{\it Proc. Nat. Acad. Sci. U.S.A.} {\bf 93} (1996), no. 4, 1659--1664. \texttt{doi:10.1073/pnas.93.4.1659}

\bibitem{figa} A. Fig\`{a}-Talamanca, An application of Gelfand pairs to a problem of diffusion in compact ultrametric spaces. {\it Topics in Probability and Lie Groups: Boundary Theory}, 51--67, CRM Proc. Lecture Notes, vol. 28, {\it Amer. Math. Soc., Providence, RI}, 2001.

\bibitem{cuttingshuffling} S.W. Golomb, Permutations by cutting and shuffling, {\it SIAM Rev.} {\bf 3} (1961), 293--297.

\bibitem{knuth} D.E. Knuth, The Art of Computer Programming, Vol.2 : Seminumerical Algorithms. {\it Addison-Wesley Publishing Co., Reading, Mass.-London-Don Mills, Ont} 1969 xi+624 pp.

\bibitem{medvedoffmorrison} S. Medvedoff, K. Morrison, Groups of perfect shuffles, {\it Math. Mag.} {\bf 60} (1987), no. 1, 3--14.

\bibitem{rose} D.J. Rose, Matrix Identities of the Fast Fourier Transform, {\it Linear Algebra Appl.} {\bf 29} (1980), 423--443. \texttt{doi:10.1016/0024-3795(80)90253-0}

\bibitem{ronse} C. Ronse, A generalization of the perfect shuffle, {\it Discrete Math.} {\bf 47} (1983), no. 2--3, 293--306. \texttt{doi:10.1016/0012-365X(83)90100-0}

\bibitem{segnali} J.O. Smith III, Mathematics of the Discrete Fourier Transform (DFT): with Audio Applications, Second Edition. W3K Publishing, 2007. 322 pp.

\bibitem{strang} G. Strang, Wavelets, {\it American Scientist} {\bf 82} (1994), no. 3, 250--255.
\end{thebibliography}
\end{document}